\documentclass[12pt]{article}
\usepackage[english]{babel}
\usepackage[a4paper,top=2cm,bottom=2cm,left=2cm,right=2cm,marginparwidth=1.75cm]{geometry}

\usepackage{amssymb}
\usepackage{amsmath}
\usepackage{amsthm}
\usepackage{graphicx}
\usepackage{xcolor}
\usepackage[colorlinks=true, allcolors=black]{hyperref}
\usepackage{mathtools}
\newcommand{\overbar}[1]{\mkern
1.5mu\overline{\mkern-1.5mu#1\mkern-1.5mu}\mkern 1.5mu}
\usepackage{enumerate}

\newtheorem{theorem}{Theorem}[section]

\newtheorem{COROLLARY}{Corollary}[section]

\newtheorem{LEMMA}{Lemma}[section]
\newtheorem{REMARK}{Remark}[section]
\newtheorem{DEFINITION}{Definition}[section]
\numberwithin{equation}{section}

\usepackage{xcolor}
\usepackage{titlesec}
\newcommand{\inlinesubsection}[1]{%
  \refstepcounter{subsection}
  \addcontentsline{toc}{subsection}{\protect\numberline{\thesubsection}#1}
  \textbf{\thesubsection\ #1. }
}
\newtheorem*{Ex}{Examples}

\DeclareMathOperator{\riem}{\mathrm{Rm}}
\DeclareMathOperator{\ric}{\mathrm{Ric}}
\DeclareMathOperator{\vol}{\mathrm{vol}}

\begin{document}
\title{\textbf{$L^1$ curvature bounds for Type I Ricci flows}} 

\author{Panagiotis Gianniotis\footnote{Department of Mathematics, National and Kapodistrian University of Athens, Greece, \qquad\qquad\qquad\qquad pgianniotis@math.uoa.gr } \and Konstantinos Leskas\footnote{Department of Mathematics, National and Kapodistrian University of Athens, Greece, kostasleskas@math.uoa.gr}}
\date{}

\maketitle

\begin{abstract} \noindent 
We show $L^1$--bounds of the Riemann curvature tensor on a smooth closed $n$--dimensional Ricci flow. To achieve this we introduce the notion of a neck of maximal symmetry, similar to the one in Cheeger--Jiang--Naber and Jiang--Naber and establish a decomposition result by balls with uniform curvature bounds that satisfy an appropriate $(n-2)$--content estimate. 
\end{abstract}


\section{Introduction}
A smooth $1$--parameter family $g(t)$, $t\in [0,T)$ of Riemannian metrics on a closed (compact with no boundary) manifold $M$ with $n=\dim M$ is said to evolve under Ricci flow if
\begin{equation*}
    \frac{\partial g}{\partial t} = -2\ric_{g(t)}.
\end{equation*}
Typically, a smooth Ricci flow will encounter finite time singularities, namely $T<+\infty$ and there is no smooth continuation of $g(t)$ after $t=T$, in  which case the norm of the curvature tensor blows up, namely
\begin{equation*}
\lim_{t\rightarrow T} \max_M |\riem(g(t))|_{g(t)} = +\infty.
\end{equation*}
Understanding the structure of the geometry of the evolving manifold as singularities develop lies at the center of potential applications of Ricci flow in Geometry and Topology.

In dimension three, after the work of Perelman \cite{Per_02, Per2,Per3} it is well understood that most local singularities are modeled on the shrinking cylinder Ricci flow on $\mathbb R \times \mathbb S^2$, or its quotients. Moreover, tangent flows at the singular points converge to self--similar shrinking Ricci flows, induced by shrinking Ricci solitons. For type--I singularities this was established by Naber \cite{Naber}, Enders--Buzano--Topping \cite{Enders} and Buzano--Mantegazza \cite{MantMul}, while for general singularities it was achieved by the compactness and structure theory of Bamler \cite{Bamler_entropy, Bamler_comp, Bamler}, see also the recent work of Fang--Li \cite{Fang2}. Although in dimension three shrinking Ricci solitons are classified, these are just the standard soliton structure on $\mathbb R^3$, as well as in $\mathbb S^3$ and $\mathbb R\times \mathbb S^2$ and their quotients, such classification is widely open in higher dimensions and in fact it is possible that the limit solitons are singular.

It is however possible to classify the singularity models according to the number of Euclidean factors they split. Since there is no, non-trivial, shrinking Ricci soliton splitting more than $n-2$ Euclidean factors, Perelman's pseudolocality theorem then implies that shrinking solitons that arise as blow-up limits at the singular points should split at most $n-2$ Euclidean factors. In this case, the only (orientable) possibility is the shrinking cylinder soliton in $\mathbb R^{n-2}\times \mathbb S^2$, since the only orientable 2d shrinking soliton is the shrinking sphere. 

In terms of their volume, it is expected that the most prevalent singularities of Ricci flow are modelled on these solitons. A typical example is that of a product Ricci flow on $\mathbb S^2\times T^{n-2}$, where $T^{n-2}$ is a flat torus, becoming singular at $t=0$. In this example, all the tangent flows are given by the self--similar shrinking Ricci flow on $ \mathbb R^{n-2}\times \mathbb S^2$. Moreover, its volume decays at a rate $\sim |t|$  while the curvature blows-up at a rate $\sim |t|^{-1}$. It is thus expected that given a smooth closed Ricci flow $(M,g(t))_{t\in [0,T)}$, the volume of regions with $|\riem|\gtrsim r^{-2}$ will also exhibit similar behaviour, namely
$$\vol_{g(t)}\left( \{|\riem_{g(t)}| \gtrsim r^{-2} \} \right) \lesssim r^2,$$
and even that for every $t\in [0,T)$
$$\int_M |\riem |(\cdot,t) d\vol_{g(t)} \leq C,$$
for some uniform constant $C<+\infty$.

In this paper, we establish that this is indeed the case, if $(M,g(t))_{t\in [0,T)}$ satisfies a type--I bound on the curvature, namely there is a constant $C_I$ such that
\begin{equation}
|\riem|\leq \frac{C_I}{T-t}
\end{equation}
on $M\times [0,T)$.

\begin{theorem}\label{intro:thm}
Let $(M,g(t))_{t\in [0,T)}$ be a smooth closed Ricci flow satisfying 
\begin{align}
|\riem|&\leq \frac{C_I}{T-t},\label{eqn:intro_type_i}\\
\nu(g(0),2T)&\geq -\Lambda. \label{eqn:intro_entropy}
\end{align}
There is a constant $C=C(n,C_I,\Lambda,\vol_{g(0)}(M),T)<+\infty$ such that
\begin{equation}\label{eqn:intro_L1}
\int_M |\riem|(\cdot,t) d\vol_{g(t)}\leq C,
\end{equation}
for every $t\in [0,T)$.
\end{theorem}

Theorem \ref{intro:thm} was first established in \cite{PanGian2} for 3d type--I Ricci flows, where estimate \eqref{eqn:intro_L1} also suffices to establish Perelman's bounded diameter conjecture in that setting, by the result of Topping in \cite{Top}. Unfortunately, in higher dimensions the $L^1$--bound on the curvature tensor is not enough to control the diameter of the evolving manifold. Moreover, Theorem \ref{intro:thm} improves the $L^p$ bounds, for $p<1$, first established in the type--I case in \cite{PanGian3} and recently for non--collapsed Ricci flow limit spaces in \cite{Fang2}. In fact we prove a stronger integral estimate, for the curvature radius, defined below, in Theorem \ref{curv radius int}. \\

Theorem \ref{curv radius int} is a consequence of our $b$--ball decomposition Theorem \ref{b-ball dec}. To describe the result, let $(M,g(t))_{t\in [-2,0]}$ be a Ricci flow and $R<+\infty$ a large positive number. We denote
$$\tilde B_R(x,r)= B(x,-r^2, Rr)=B_{g(-r^2)}(x,Rr).$$
and define the curvature radius $r^R_{\riem}(x)$ as
$$r^R_{\riem}(x)=\sup\left\{ r>0, \textrm{$|\riem|\leq r^{-2}$ in $\tilde B_R(x,r)\times [-r^2,0]$}\right\}.$$

\begin{theorem}[Theorem \ref{b-ball dec}] \label{intro:thm_b_ball}
Let $(M,g(t),p)_{t\in [-2,0]}$ be a pointed smooth closed Ricci flow satisfying \eqref{eqn:intro_type_i} and \eqref{eqn:intro_entropy}, and let $R\geq R(n,C_I,\Lambda)$. Then there is a covering of the form
\begin{equation*}
\tilde B_R(p,1) \subset \bigcup_b \tilde B_R(x_b,r_b),
\end{equation*}
where $r_{\riem}^{2R}(x_b) >2r_b$, $r_b<1$ and 
\begin{equation*}
\sum_{b} r_b^{n-2} \leq C(n,C_I,\Lambda,R).
\end{equation*}
\end{theorem}
 The proof of Theorem \ref{intro:thm_b_ball} follows an iterative covering scheme similar to that in \cite{JiangNaber, CJN}, and relies on the neck structure theorem obtain for type--I Ricci flows in \cite{PanGian2}, see Theorem \ref{neck structure}. Roughly speaking, this scheme iteratively refines covers of $\tilde B_R(p,1)$ with balls of consecutively smaller scales. At each step, the balls of each such cover are separated into various categories by quantifying whether the flow is close to a shrinking cylinder $\mathbb R^{n-2}\times \mathbb S^2$ or not. This process continues until it terminates with the balls $\tilde B_R(x_b,r_b)$ with  $2r_b<r_{\riem}^{2R}(x)$. Regions in the flow with this cylindrical behaviour over many scales can be eventually described as $(n-2)$--neck regions, in the sense of \cite{PanGian2} and Definition \ref{def:necks_ms}, and it is for these regions that the neck structure Theorem \ref{neck structure} provides the necessary content estimates. The whole covering argument is implemented in a series of Lemmata, in particular Lemma \ref{d-ball dec}, Lemma \ref{neck construction/estimates}, Lemma \ref{inductive covering} and Lemma \ref{v-ball dec}. 

A crucial concept in this process is the \textit{entropy pinching} set around $x\in M$, defined as
$$\mathcal P_{r,\xi,R}(x)=\{y\in \tilde B_{R/2} (x,r) | \mathcal W_{y}(\xi r^2) - \overbar W<\xi\},$$
where $\mathcal W_y(\tau)$ denotes Perelman's pointed entropy, see \eqref{pointed entropy} for the definition. Here, $\xi>0$ is a small constant, whereas the constant $\overbar W\leq 0$ should be thought to satisfy $\overbar W\leq \inf_{\tilde B_{2R}(x,1)} \mathcal W_y(\xi^{-1})$, so that the entropy pinching set is designed to detect points where the pointed entropy is close to $\overbar W$ over many scales.  It turns out that to detect balls with the required cylindrical behaviour it is enough to search for balls with pinching set of large enough Minkowski $(n-2)$--content. This is the content of Theorem \ref{q.split}, the quantitative splitting theorem. Then, Lemma \ref{d-ball dec} treats balls whose pinching set has small Minkowski $(n-2)$--content, while Lemma \ref{neck construction/estimates} constructs neck regions in those balls whose pinching set has large Minkowski $(n-2)$--content. The combination of these two lemmata leads to a covering of a given $\tilde B_R(p,1)$ by balls $\tilde B_R(x_b,r_b)$, as in the statement of Theorem \ref{b-ball dec}, together with balls with empty entropy pinching set. Lemma \ref{inductive covering} and Lemma \ref{v-ball dec} then show that these balls can be eliminated from the final cover, essentially because they imply a small but fixed jump in the pointed entropy, which can only happen on finitely many scales.\\

Both \cite{PanGian2} and this work rely on the neck structure Theorem \ref{neck structure}, established in \cite{PanGian2}, and a neck decomposition, akin to those in  \cite{JiangNaber, CJN}. To our knowledge, this work and \cite{PanGian2} are the first instances in which these techniques have been applied to the Ricci flow. In the context of mean curvature flow of hypersurfaces in $\mathbb R^{n+1}$, neck decompositions in the spirit of \cite{JiangNaber, CJN} have been applied by Fang--Li \cite{Fang} to establish sharp volume estimates for the singular and quantitative singular sets in the mean convex case, and recently by Huang--Jiang in \cite{Huang} to prove the bounded (intrinsic) diameter conjecture for surfaces in $\mathbb R^3$, without any assumptions on the nature of the singularities that may develop. The question in higher dimensions remains however open in such generality, and only understood when only cylindrical or conical singularities develop \cite{PanGian4,Du}.

In the results \cite{PanGian4,Du,Fang,Huang} a crucial ingredient is the \L ojasiewicz inequality for cylindrical shrinkers by Colding--Minicozzi \cite{Col-Min}, and not a neck structure result, as in our work and \cite{PanGian2}. Such an inequality for cylindrical Ricci solitons for the Ricci flow was only recently established in the remarkable preprint \cite{Fang1} by Fang--Li. It is likely that such a result, combined with the decomposition ideas developed in our work as well as in \cite{Fang}, and the weak compactness and structure theory for the Ricci flow by Bamler \cite{Bamler_entropy,Bamler_comp, Bamler} and Fang--Li \cite{Fang2}, will be able to treat the case of any type of cylindrical singularities for the Ricci flow, without the need of a type--I curvature bound. However, the neck structure theorem in \cite{PanGian2}, although only established under the type--I assumption, completely bypasses the need for a \L ojasiewicz inequality and can give information for any singularity model. We thus expect that it can be applied to understand the structure of the full singular set, with no assumption on the type of singularities, and in any dimension. We will treat this general case in forthcoming work.

\medskip

\noindent {\bf Acknowledgments:}  This research was supported by the Hellenic Foundation for Research and Innovation (H.F.R.I.) under the “2nd Call for H.F.R.I. Research Projects to support
Faculty Members \& Researchers” (Project Number: HFRI-FM20-2958).

\section{Preliminaries} 

In this section we recall some necessary preliminaries for smooth closed (compact with no boundary) Ricci flows and set up some notation.

\medskip

\noindent
\inlinesubsection{Pointed entropy, a priori assumptions and $R$--scale distance} For a smooth closed Ricci flow $(M, g(t))_{[-T, 0]}$ we say that the function $u(\cdot , t) =(4\pi|t|)^{-n/2} e^{-f(\cdot , t)} $ evolves by the conjugate heat equation if 
\begin{equation}
\label{conj heat}
\frac{\partial u}{\partial t} = - \Delta_{g(t)} u + \text{Scal}_{g(t)}u, 
\end{equation}
where $\text{Scal}_{g(t)}$ denotes the scalar curvature with respect to the metric $g(t)$. Equivalently 
\begin{equation}
\label{equiv conj heat}
\frac{\partial f}{\partial t} = -\Delta_{g(t)}f+|\nabla f|_{g(t)}^2- \text{Scal}_{g(t)} + \frac{n}{2|t|}.
\end{equation} 
We denote by $u_{(p,0)}$ the conjugate heat kernel based at $(p,0)$. This means that $u_{(p,0)} >0$, satisfies the conjugate heat equation \eqref{conj heat} and furthermore 

$$\int_M u_{(p,0)}(x,t) \phi(x) dV_{g(t)} (x) \xrightarrow[t \to 0]{} \phi(p),$$
for any $\phi \in C^\infty(M)$ and for all $t$
$$\int_M u_{(p,0)}(x,t) dV_{g(t)}(x) = 1.$$ 
For the existence and uniqueness of the conjugate heat kernel based at $(p,0)$ one can see  \cite[Chapter 24]{Chow II}. A notion relevant to the conjugate heat kernel is that of a conjugate heat flow. We denote by $(\nu_{f,t})$ a conjugate heat flow of the flow $(M, g(t))_{[-T,0]}$ if $d\nu_{f,t} = (4\pi|t|)^{-n/2} e^{-f(\cdot,t)}dV_{g(t)}$ is a probability measure for each $t<0$, where the function $u(\cdot,t) = (4\pi|t|)^{-n/2} e^{-f(\cdot,t)} $ is a positive solution to \eqref{conj heat} and $\int_{M}u(\cdot,t)dV_{g(t)} = 1$ for any $t$. If the density of the conjugate heat flow is a conjugate heat kernel based at $(p,0)$ we will also write $\nu_{(p,0)}.$

From the seminal work of Perelman \cite{Per_02}, we can define the entropy functional of a closed Riemannian manifold $M$ as 
\begin{equation*}
    \mathcal W(g, f, \tau) =  \int_M  \bigg( \tau ( \text{Scal} + |\nabla f|^2 ) + f - n  \bigg)\dfrac{e^{-f}}{(4\pi\tau)^{n/2}}dV_g,
\end{equation*} 
where $\tau>0, f \in C^\infty(M)$. We also define the functionals 

\begin{equation*}
    \mu(g, \tau) = \inf_{f\in C^\infty(M)}\bigg\{\mathcal W(g, f, \tau) \bigg |  \: \int_M \dfrac{e^{-f}}{(4\pi\tau)^{n/2}} dV_g = 1\bigg\},
\end{equation*}

\begin{equation*}
\nu(g, \tau) = \inf_{0< \tau' < \tau} \mu(g, \tau').    
\end{equation*}
Note that $\mu$ is finite for each $\tau>0$, by the celebrated logarithmic Sobolev inequality for closed Riemannian manifolds, see \cite{Per_02} or \cite[Chapter 6]{Chow I}.

From the monotonicity formula proven in \cite{Per_02} we have that

\begin{equation}
\label{monotonicity}
    \dfrac{d}{dt} \mathcal W(g(t), f(\cdot , t), |t|) = 2|t|\int_M \bigg | Ric_{g(t)} + \nabla^2_{g(t)} f - \dfrac{g}{2|t|}  \bigg|^2 u(\cdot, t) \: dV_{g(t)} \geq 0,
\end{equation}
where $u(\cdot, t) = (4\pi |t|)^{-n/2}e^{-f(\cdot , t)}$ evolves by the conjugate heat equation \eqref{conj heat}. In particular, we have that for $t_1, t_2 \in[-T, 0]$ with $t_1 < t_2$ the functionals $\mu, \nu$ are non--decreasing
$$\mu(g(t_1), |t_1|) \leq \mu(g(t_2), |t_2|),$$ 
 $$\nu(g(t_1), |t_1|) \leq \nu(g(t_2), |t_2|).$$ 

From \cite{MantMul} we can define the pointed entropy functional of a smooth closed Ricci flow $(M, g(t))_{[-T, 0]}$ as follows. Let $u_{(p,0)}(x,t) =(4\pi|t|)^{-n/2} e^{-f(x, t)} $ be the conjugate heat kernel based at $(p,0)$. Then the pointed entropy at the point $p\in M$ and scale $|t|$ is given by
\begin{equation}
\label{pointed entropy}
\mathcal W_p(|t|) = \mathcal W(g(t), f( \cdot , t), |t| ).
\end{equation}

\noindent 
From \eqref{monotonicity} we deduce that 
\begin{equation}
\text{ if } |t_1| < |t_2| \text{ then } \mathcal W_p(|t_1|) \geq \mathcal W_p (|t_2|)
\end{equation} 
and from Perelman's Harnack inequality, proven in \cite{Per_02}, see also  \cite[Theorem 16.44]{Chow II}, we have that 
\begin{equation}
\label{PH in}
    \mathcal W_p(|t|) \leq 0,  \: \text{ for all t $\leq$ 0.} 
\end{equation}

\medskip 
\noindent
We will be interested in smooth, closed Ricci flows $(M, g(t))_{[-T, 0]}$ that satisfy the following: 

\begin{enumerate}
    \item[RF1)]\label{RF1} $\displaystyle \sup_{x\in M} |Rm|_{g(t)}(x, t) \leq \dfrac{C_I}{|t|} $, for all $t\in[-T,0)$ for some $C_I >0.$ 

    \item[RF2)]\label{RF2} $\nu(g(-T), 2T)\geq - \Lambda,$ for some $\Lambda>0.$ 
    \end{enumerate}

\noindent
We will denote this class of Ricci flows by $\mathcal{F}(n, C_I, \Lambda).$ We have the following: 

\begin{LEMMA}\label{properties} Let $(M, g(t))_{[-T,0]}\in \mathcal{F}(n,C_I, \Lambda)$ then
\begin{enumerate}
     \item\label{1} $\nu(g(t), |t|) \geq -\Lambda$, for all $t \in [-T, 0].$ 
     \item\label{2} If $u_{(x,t)}(y,s) = \dfrac{e^{-f(y,s)}}{(4\pi|t-s|)^{n/2}}$, for $s<t$, is a conjugate heat kernel based at $(x,t)$ then 
    $$\dfrac{d^2_{g(s)}(x,y)}{H |t-s|} - H \leq f(y,s) \leq H\bigg( \dfrac{d^2_{g(s)}(x,y)}{|t-s|} +1 \bigg).$$ 
     \item\label{3} $d_{g(s)} (x,y) \leq K(\sqrt{|s|}- \sqrt{|t|} )+d_{g(t)}(x,y)$, for $s<t<0$, 
     \item\label{4}$d_{g(t_1)}(x,y) \leq \bigg(\frac{|t_2|}{|t_1|}\bigg)^{C_I} d_{g(t_2)}(x,y)$ for any $t_2<t_1<0$
\end{enumerate} 
 where $H$ is a constant that depends on the dimension $n$ and the constants $C_I, \Lambda$ and $K$ a constant that depends on $n, C_I$   
 \end{LEMMA}

 \begin{proof} Property \ref{1} is an immediate application of the monotonicity of the $\nu$ functional. Property \ref{2} is a consequence of Proposition 2.7 and 2.8 of \cite{MantMul}. Property \eqref{3} is a consequence of \cite[Lemma 2.6]{Naber}, see also \cite[Lemma 8.3]{Per_02}. We briefly explain the latter. Let $\gamma$ be a unit speed minimizing geodesic with respect to the metric $g(t)$, between the points $x, y$. Then 
 $$\int_\gamma Ric_{g(t)}(\gamma'(s), \gamma'(s))ds \leq \frac{C}{\sqrt{|t|}}.$$ 
 From the Ricci flow equation we obtain that 
 $$\frac{d}{dt} d_{g(t)}(x,y) \geq - \frac{C}{\sqrt{|t|}},$$
 where $C$ is a positive constant that depends on $n, C_I$. Integrating the above inequality gives the desired estimate.

 Finally we show property \eqref{4}. Fix $v$ a unit vector field then using the Ricci flow equation and the type--I bound we have 
 $$\log\bigg(\dfrac{g(t_1)(v,v)}{g(t_2)(v,v)}\bigg) = \int_{t_2}^{t_1}\frac{\partial}{\partial t} \log \bigg(\dfrac{g(t)(v,v)}{g(t_2)(v,v)}dt\bigg) \leq \int_{t_2}^{t_1}\bigg|\frac{\partial}{\partial t} g(t)\bigg|(v,v)dt=$$
 $$= \int_{t_2}^{t_1}2|Ric_{g(t)}|(v,v)dt\leq \int_{t_2}^{t_1}\frac{2C_I}{|t|}dt = \log\bigg(\frac{|t_2|}{|t_1|}\bigg)^{2C_I},$$
 which shows that $g(t_2)(v,v) \leq \bigg(\frac{|t_2|}{|t_1|}\bigg)^{2C_I}  g(t_1)(v,v)$.
 Choose a minimizing unit speed geodesic from $x$ to $y$ with respect to the metric $g(t_2)$ then for $v = \dot \gamma(t)$ and integrating the previous inequality we get \eqref{4}.
 \end{proof}

\begin{REMARK} We will also use the weaker inequality 
\begin{equation}\label{weak dist dist}
d_{g(s)}(x,y) \leq d_{g(t)}(x,y) + K \sqrt{|s|}, \text{ for any } s<t<0.
\end{equation} The latter is a consequence of \cite[Theorem 1.1]{B-Z II}, see also \cite[Proposition 3.2]{PanGian}. Both inequalities will be referred to as distance distortion estimates. We highlight that for the distance distortion of Lemma \ref{properties}\eqref{3} the type--I bound on the curvature tensor is essential, while for \eqref{weak dist dist}  we only need a type--I bound on the scalar curvature. 
\end{REMARK}

\medskip

We will make use of the $R$--scale distance function, defined in \cite{PanGian2}. For a flow in the class $\mathcal{F}(n,C_I, \Lambda)$ we define for $R >0$
$$D_R(x,y) = \inf \{r \in (0, \sqrt{T}] \mid d_{g(-r^2)}(x,y) < r R \},$$
and we denote by $\tilde B_R (x, r)$ a ball with respect to $D_R.$ So $y \in \tilde B_R (x,r)$ if and only if $D_R(x,y) < r.$ Note that if the set over which we take the infimum is empty then we set $D_R(x,y=+infty.$  From Lemma \ref{properties}\eqref{3} we have the following lemma (see \cite[Proposition 3.2]{PanGian2}).

\begin{LEMMA}\label{properties of D_R} Let $(M, g(t))_{[-T,0]}\in \mathcal{F}(n,C_I, \Lambda)$ then the $R$--scale distance function satisfies the following properties 
\begin{enumerate}
\item\label{almost triangle}
    For $R_1, R_2 \geq R-K,$ \: $ D_R(x,y) \leq \dfrac{\max (R_1, R_2)}{R-K} ( D_{R_1}(x,z) + D_{R_2}(z,y) ), $ for any $x,y,z \in M.$ 
    \item\label{sd1} For $R_1 \leq R_2,$ \:$D_{R_2}(x,y) \leq D_{R_1}(x,y) \leq \dfrac{R_2-K}{R_1-K} D_{R_2}(x,y),$ for any $x,y \in M.$ 
    \item\label{inclusion} If $r_1 < r_2$  then $ \tilde B_R(x,r_1) \subset \tilde B_R(x, r_2).$ 
    \end{enumerate}
\end{LEMMA}

\noindent
\inlinesubsection{Volume estimates on scale distance balls}
Let $(M, g(t))_{[-2, 0]} \in \mathcal{F}(n,C_I, \Lambda)$ then if $R r < 1$, $r<1$, the non--inflating estimate established in \cite{Zhang}, together with a covering argument, see \cite[Lemma 2.1]{B-Z II},  implies that there exists a constant $v_2:= v_2(n, C_I, \Lambda)$ such that

\begin{equation}
\label{non-inf}
    \vol_{g(-r^2)}(\tilde B_R (x, r)) \leq v_2 r^n R^n.
\end{equation}

\begin{REMARK} If we have a flow $(M, g(t))_{[-2\xi^{-1},0]} \in \mathcal F(n, C_I, \Lambda)$ then a simple rescaling argument implies that the estimate \eqref{non-inf} holds for any $R>0$ and $r \leq 1$, provided that $\xi \leq \xi (R).$  
We will use both versions of the non--inflating estimate.
    
\end{REMARK}

We recall now the non--collapsing theorem of \cite{Per_02}, improved in \cite{Top}. On a closed Riemannian manifold let $B(p,r)$ be a ball where $\text{Scal} \leq r^{-2}$. Then
\begin{equation}
\label{non-coll I}
\vol(B(p,r)) s^{-n} \geq C(n) e^{\nu(g,r^2)},
\end{equation}
for every $s\in (0, r].$ 

In particular, on a closed Ricci flow $(M, g(t))_{[-2,0]} \in \mathcal{F}(n,C_I, \Lambda)$ let $R > R(C_I)$, $r\leq 1$, then in $\tilde B_{\frac{1}{\sqrt{C_I}}}(x,r) \subset \tilde B_R (x,r)$ we have that $|Rm|_{g(-r^2)} \leq \frac{1}{(r C_I^{-1/2})^2}$. Thus from \eqref{non-coll I} there exists $v_1:= v_1(n, C_I, \Lambda)$ such that 

\begin{equation}
\label{non-coll}
    v_1 r^n \leq \vol _{g(-r^2)}(\tilde B_R(x, r)). 
\end{equation}

We also recall the following volume comparison which is valid on any closed Ricci flow $(M, g(t))_{[-T, 0]}$. From the differential inequality for the scalar curvature, 
$$\frac{\partial}{\partial t}  \text{Scal}_{g(t)} \geq \Delta_{g(t)} \text{Scal}_{g(t)} + \frac{2}{n}\text{Scal}^2_{g(t)},$$ a standard application of the maximum principle gives that 
$\text{Scal}_{g(t)}(\cdot, t) > -\frac{n}{2(T +t)}.$ 
In particular if $T\geq2$ then for $|t|\leq 1 $ we get that 
$$\text{Scal}_{g(t)}(\cdot, t) > -\frac{n}{2}.$$
Thus from the evolution equation of the volume $V(t) = \vol_{g(t)}(A)$,  of a set $A \subset M$ 
$$\frac{d}{dt} log V(t) = - \text{Scal}_{g(t)},$$ 
we get the following volume comparison

\begin{equation}
    \label{volume comp}
    V(t_2) \leq C(n) V(t_1),
\end{equation}
for $t_1<t_2 \leq0 $, $|t_1| \leq 1$, where $C(n)$ denotes a constant that depends on the dimension $n$.

\medskip

\noindent
\inlinesubsection{Almost self--similar Ricci flows} We start by recalling some facts about gradient shrinking Ricci solitons. We say that the triplet $(M, g_0, f)$, where $(M, g_0)$ is a Riemannian manifold and $M$ is complete, is a gradient shrinking Ricci soliton at scale $\tau$ if it satisfies the PDE
$$Ric_{g_0} + \nabla^2 f = \frac{g_0}{2\tau}.  $$ 

The bounds from \cite[Lemma 2.1]{HassMul} together with the volume growth of \cite{Car} imply that 
$$\int_M (4\pi \tau)^{-n/2}e^{-f}dV_{g_{0}} < \infty, $$
thus we may always normalize so that the above integral is 1. We will then call $(M, g_0, f)$ a normalized gradient shrinking soliton. The entropy of a normalized gradient shrinking soliton $(M, g_0, f)$ is defined as $\mu(g_0) =  \mathcal{W}(g_0, f, \tau )$, where $\mathcal{W}$ is Perelman's entropy functional. The latter is well--defined even if $M$ is non--compact due to the estimates of \cite{HassMul} and \cite{Car} and it depends solely on the metric $g_0.$

Let $(M, g_0, f)$ be a normalized gradient shrinking soliton at scale 1 and $\phi_t: M \to M$ the family of diffeomorphisms defined by the equation $\dfrac{d}{dt}\phi_t = \nabla f\circ \phi_t.$ Then $(M, g_0, f)$ induces the Ricci flow $(M, g(t))_{(-\infty,0)}$, where $g(t) = |t|\phi_t^* g_0$. For each $t<0$ and $f_t = f\circ \phi_t,$ $(M, g(t), f_t)$ is a normalized gradient shrinking soliton at scale $|t|.$  Note that $\mu(g(t)) = \mu(g_0)$ due to the scale invariance of the functional.

\begin{DEFINITION}Let $(M, g(t))_{(-\infty,0)}$ be a Ricci flow induced by a gradient shrinking soliton. The spine $\mathcal{S}$ of the flow is the collection of conjugate heat flows $(\nu_{f,t})$ such that $(M, g(t), f(\cdot,t))$ is a normalized gradient shrinking soliton at scale $|t|$.  The point spine is $$S_{point} =\bigg \{ x \in M \mid \exists \: \nu_f \in \mathcal S \text{ and } f(\cdot, t) \text{ attains a minimum at } x \text{ for every } t<0 \bigg\}.$$ 
\end{DEFINITION}

\noindent
Thus the point spine consists of the minimum points of the soliton maps, hence $S_{point}\neq \emptyset$ from the work of \cite{HassMul}. 

We will call $(M, g(t), p)_{(-\infty,0)}$ a $k$--self--similar flow if it is induced by a gradient shrinking soliton, $p \in S_{\text{point}}$ and is isometric to $(M'\times \mathbb R^k, g'(t)\oplus g_{\mathbb R^k})_{(-\infty,0)}$ for each $t<0,$ where $(M',g'(t))_{(-\infty,0)}$ is a flow induced by a gradient shrinking soliton. Note that if the flow is $k$--self--similar then there exists $l \geq k$ such that 
$S_{point} = K \times \mathbb R^l, $ $K \neq \emptyset$ and 
$$\text{diam}_{g'(t)} (K) \leq A(n) \sqrt{|t|}.$$ The latter inequality follows from \cite{HassMul}. See also \cite[Theorem 4.1]{PanGian} for further properties of $k$--self--similar flows. 

\begin{Ex}If the flow is the shrinking sphere $\mathbb S^n$ then it is $0$--selfsimilar (with respect to any point) and $S_{point} = \mathbb S^n$ since the potential map $f$ is a constant. Similarly the shrinking cylinder $\mathbb R ^k \times \mathbb S^{n-k}$ is $k$--self--similar and we have that $S_{point } = \mathbb R ^k \times \mathbb  S^{n-k}$. On the other hand in the shrinking sphere the spine $\mathcal{S}$ consists of only one element while in the shrinking cylinder $\mathcal{S}$ is $k$--dimensional. 
\end{Ex}

\begin{REMARK}The above examples show that the definition of the point spine is in agreement with the fact that in the shrinking sphere or the shrinking cylinder every point of the flow is a singular point. This is unlike the case of minimal cones or shrinkers on a mean curvature flow where the point spine is defined to be a $k$--dimensional subspace. 
\end{REMARK}

\noindent 
Self--similar flows arise as pointed blow--up limits of type--I compact Ricci flows, see \cite{Naber}, \cite{Enders}. We refer the reader also to the work of \cite{Bamler} where blow--up limits of Ricci flows are studied with no type--I bound on the curvature tensor.

\begin{DEFINITION} Let $(M, g(t), p)_{[-2\delta^{-1}r^2, 0]}$ a pointed Ricci flow. Then we say that it is $(k, \delta)$--self--similar at scale $r$ if for $g_r(t) = r^{-2}g(r^2t)$ there exists a $k$--self--similar flow $(\tilde M, \tilde g(t), \tilde p)_{(-\infty, 0)}$ and $F: B_{\tilde g(-1)}(\tilde p, \delta^{-1}) \to M$ a diffeomorphism to its image, with $F(\tilde p)=p$ ($\tilde p \in S_{point}$) such that in $B_{\tilde g (-1)}(\tilde p , \delta^{-1})\times [-\delta^{-1}, -\delta]$ we have
$$||F^\star g_r - \tilde g||_{C^l} < \delta, $$
for any $l < \delta^{-1}$ and 
$$|d_{g_r(t)}(p, F(x)) - d_{\tilde g (t)}(\tilde p, x)| < \delta.$$
We define the approximate singular set $\mathcal{L}_{p , r} = F(B_{\tilde g(-1)}(\tilde p , \delta^{-1})\cap S_{\text{point}})$ and say that the flow is $(k, \delta)$--self--similar around $p$ at scale $r$ with respect to $\mathcal{L}_{p, r}.$ Moreover, $(\tilde M, \tilde g(t)), \tilde p)_{(-\infty,0)}$ will be called a model flow for the flow $(M, g(t), p)_{[-2\delta^{-1},0]}.$ 
    
\end{DEFINITION}

We close this section by mentioning the necessary compactness results. 
The convergence is in the usual Cheeger--Gromov--Hamilton sense for pointed flows. Furthermore, if a sequence of flows $(M_j, g_j(t), p_j)_{t \in I_j}$ converges to the flow $  (M_\infty, g(t)_\infty, p_\infty)_{t \in I_\infty}$ then we say that the conjugate heat flows $\nu_{f_j,t}$ converge to a conjugate heat flow $\nu_{f,t}$ if the associated densities of the measures converge locally in $C^\infty(M_\infty\times I_\infty).$

\begin{theorem}\textup{(compactness and entropy continuity)} \label{compactness}
\begin{enumerate}
\item\label{I} Let $(M_j, g_j(t), p_j)_{t \in I_j} \in \mathcal{F}(n,C_I, \Lambda)$, a sequence of pointed flows. Then, after passing to a subsequence, it converges to a smooth, complete flow $(M_\infty, g_\infty(t), p_\infty)_{t \in I_ \infty}$ that satisfies $RF1, RF2.$ 
\item\label{II} Let $(M_j, g_j(t), p_j)_{[-2\delta_j^{-1}, 0]} \in \mathcal{F}(n,C_I, \Lambda)$ be a sequence of $(k, \delta_j)$--self--similar flows at scale 1 then as $\delta_j \to 0$ and up to a subsequence there exists a limit flow $(M_\infty, g_\infty(t), p_\infty)_{(-\infty,0)}$ that is $k$--self--similar and for any $t<0$ we have that $\mathcal W_{p_j}(|t|) \to \mu(g_\infty(-1))$, for all $t<0.$ 

\item\label{III}Let $(M_j, g_j(t), p_j)_{t \in I_j} \in \mathcal{F}(n,C_I, \Lambda)$ with a limit flow $(M_\infty, g_\infty(t), p_\infty)_{(-\infty,0)}$ induced by a gradient shrinking soliton and assume that for the conjugate heat flows $\nu_{(p_j,0)}$ based at $(p_j, 0)$ we have that $\nu_{(p_j,0)} \to \nu_f \in \mathcal S$. Then $\mathcal W_{p_j}(|t|) \to \mu(g_\infty(-1))$, for all $t<0.$
\end{enumerate}
\end{theorem}

Property \eqref{1} is the standard Hamilton compactness theorem, \cite{Ham1}, on Ricci flows in the class $\mathcal{F}(n, C_I, \Lambda).$ We briefly explain the entropy convergence of \eqref{II}, \eqref{III}. In \eqref{II} we have that the points $p_j$ converge to a point in $S_{\text{point}}$ thus from \cite[Proposition 4.3]{PanGian} we have that the conjugate heat flows $\nu_{(p_j,0)}$, based at $(p_j,0)$, converge to an element of $\mathcal S$. The latter implies the entropy convergence, due to \cite{MantMul} or \cite[Proposition 4.2]{PanGian}. On the other hand in \eqref{III} we already know the convergence of the conjugate heat flows based at $(p_j,0)$ thus the entropy convergence follows from \cite{MantMul} or \cite[Proposition 4.2]{PanGian}. Note that in this case we do not know if the limit point of $p_j$ belongs in the point spine of the limit soliton.

\section{Quantitative splitting and regularity} \label{section3}

In this section we define the entropy pinching set and prove a quantitative splitting result, Theorem \ref{q.split}. Furthermore, we show in Theorem \ref{eta} that an application of Perelman's pseudolocality theorem on almost self--similar flows of maximal symmetry gives us an instance of an $\eta$--regularity result.

\medskip

\noindent
\inlinesubsection{Entropy pinching set and quantitative splitting} Following the ideas of \cite{CJN} and \cite{JiangNaber} on manifolds with lower Ricci bounds and manifolds with two sided Ricci bounds respectively we give the following: 
\begin{DEFINITION} If $(M, g(t))_{[-2\xi^{-1},0]} \in \mathcal{F}(n,C_I, \Lambda)$ then we define the entropy pinching set at scale $r$ around the point $p$ as
$$\mathcal{P}_{r, \xi, R}(p) = \{y \in \tilde B_{R/2}(p, r) \mid \mathcal W_y(\xi r^2) - \overbar W < \xi \},$$
where $\overbar W$ is a fixed non--positive constant.    
\end{DEFINITION} 

\begin{LEMMA}\label{s-s/e-d} Let $(M, g(t), p)_{[-2\xi^{-1}, 0]} \in \mathcal{F}(n,C_I, \Lambda)$ and $y \in \mathcal P_{r, \xi, R}(p), $ where the pinching set is taken with respect to  $\overbar W \leq \displaystyle\inf_{\tilde B_{2R}(p,1)} \mathcal{W}_y(\xi^{-1}).$ Then for any $\delta'>0$, $R \geq R (n, C_I, \Lambda) $ and $\xi \leq \xi(n, C_I, \Lambda, \delta')$ there exists a point $x \in \tilde B_{R}(p, 1)$ such that the pointed flow $(M, g(t), x)_{[-2\delta'^{-2}r^2,0]}$ is $(0, \delta'^2)$--self--similar at scale $r$ and  $d_{g(t)}(y,x) \leq 2 D (n, C_I, \Lambda) \sqrt{|t|}$ for any $t \in [-\delta'^{-2}r^2, -\delta'^2 r^2].$ 

Furthermore, if $\gamma \in (0,1)$ and $\delta >0$ then for $\delta' \leq \delta' (R, \gamma, \delta )$ we have that for any $z \in \mathcal{L}_{x,r}\cap \tilde B_{2R}(p,r)$ 
$$\mathcal W_z(\gamma r^2 \delta) - \overbar W < \delta$$ 
and as a consequence $\mathcal W _z(\gamma r^2 \delta)-\mathcal W_z(\delta^{-1})< \delta.$

\end{LEMMA}

\begin{proof} The first part of the lemma comes from \cite[Corollary 4.2]{PanGian}, since for a flow in the class  $\mathcal{F}(n,C_I, \Lambda)$ we have conjugate heat kernel bounds, Lemma \ref{properties}\eqref{2}, needed for the proof. For the second part of the lemma we argue in a similar way. 

Assume that $\delta'_j \to 0$ and let a sequence of flows $(M_j, g_j(t),p_j)_{[-2\delta'^{-2}_j,0]}$, $y_j \in \mathcal P_{r, \xi_j, R}(p_j)$, $x_j \in \tilde B_R (p_j, r)$ such that $(M_j, g_j(t), x_j)_{[-2\delta'^{-2}_j,0]} \in \mathcal{F}(n,C_I, \Lambda)$, $(M_j, g_j(t), x_j)_{[-2\delta'^{-2}_j,0]}$ is $(0, \delta'^{2}_j)$--self--similar at scale r and there are  points $z_j \in \mathcal{L}_{x_j, r}\cap \tilde B_{2R}(p_j, r)$ such that $\mathcal W_{z_j}(\gamma r^2 \delta) - \overbar W_j > \delta.$ 

From Theorem \ref{compactness}\eqref{II} we have that $(M_j, g_j(t), z_j)_{[-2\delta'^{-2}_j,0]}$ converges to a $0$--self--similar flow given by $(M_\infty, g_\infty(t), z_\infty)_{(-\infty,0)}$ and $\mathcal W_{z_j}(|t|) \to \mu(g_\infty(-1))$ for all $t<0.$ 

From Theorem \ref{compactness} \eqref{I} we have that $(M_j, g_j(t), y_j)_{[-2\xi_j^{-1},0]}$ (note that $\xi_j \to 0)$ has a limit flow and by taking $\delta'_j \leq \delta'(R)$ we ensure that this limit is $(M_\infty, g_\infty(t), y_\infty)_{(-\infty,0)}$, however we do not know whether $y_\infty \in S_{point}$ to ensure the entropy convergence. To get the latter we argue as follows.

Since $y_j \in P_{r, \xi_j, R}(p_j)$ we have that 
\begin{equation}
\label{ensuring spine}
\xi_j > \mathcal W_{y_j}(\xi_j r^2) - \mathcal W_{y_j}(\xi_j^{-1}) = \int_{-\xi_j^{-1}}^{-\xi_j r^2}2|t| \int_{M_j} \bigg | Ric_{g_j} + \nabla^2_{g_j} f_j - \dfrac{g_j}{2|t|}  \bigg|^2 u_j dV_{g_j(t)}, 
\end{equation}
where we used the monotonicity formula \eqref{monotonicity} and $u_j$ is the conjugate heat kernel based at $(y_j, 0).$ Since $\xi_j \to 0$ and from the conjugate heat kernel bounds, Lemma \ref{properties}\eqref{2}, we conclude that the conjugate heat flows based at $(y_j,0)$ converge to a conjugate heat flow $ \nu_f$ and from \eqref{ensuring spine} we have that $\nu_f \in \mathcal{S}$. Thus from Theorem \ref{compactness}\eqref{III} we get that $\mathcal W_{y_j}(|t|) \to \mu(g_\infty(-1))$ as well. 

In conclusion we have that $\mathcal W_{y_j}(|t|)-\mathcal W_{z_j}(|t|) \to 0$. Also, $- \Lambda \leq \overbar W_j \leq 0$ thus $\overbar W_j \to  w \leq 0$ and from the monotonicity of the pointed entropy we have that $\overbar W_j \leq \mathcal W_{y_j}(\gamma r^2 \delta) \leq \mathcal W_{y_j}(\delta'_j r^2) \leq \mathcal W_{y_j}(\xi_j r^2) \leq \xi_j + \overbar W_j  $, provided $\delta'_j < \gamma \delta$. Thus $w = \mu(g_\infty(-1))$ and $\mathcal W_{z_j}(\gamma r^2 \delta)- \overbar W_j < \delta/4$ contradicting our initial assumption. Hence if $\delta' \leq \delta'(R, \gamma, \delta)$ we have that $\mathcal W_{z}(\gamma r^2 \delta) - \overbar W < \delta$ and since $\xi < \delta$ and $\overbar W \leq \displaystyle\inf_{\tilde B_{2R}(p,1) }\mathcal W_y (\xi^{-1})$ we also conclude that $\mathcal W_z(\gamma r^2 \delta) - \mathcal W_z(\delta^{-1})< \delta.$
\end{proof}

Let $A \subset M$ then the tubular neighborhood of $A$ at scale $r$ is defined to be the set $\tilde B_{R, r}(A)=\bigcup_{x \in A} \tilde B_R(x, r).$ If we have a variable radius function $r_x$ for $x \in A$ we define the tubular neighborhood of $A$ with variable radius to be $\tilde B_{R, r_x}(A) = \bigcup_{x \in A} \tilde B_{R}(x, r_x).$
\begin{theorem}\textup{(quantitative splitting at scale 1)}\label{q.split} Let $(M, g(t),p)_{[-2\xi^{-1},0]} \in \mathcal{F}(n,C_I, \Lambda)$ and $\epsilon, \delta >0.$ If $R \geq R(n, C_I, \Lambda)$, $\gamma \leq \gamma(n, C_I, \Lambda, R , \epsilon)$, $\xi \leq \xi(n, C_I, \Lambda, R, \delta, \gamma, \epsilon)$ and 
$$\vol_{g(-\gamma^2)}(\tilde B_{R, \gamma}(\mathcal P_{1, \xi, R}(p))\geq \epsilon \gamma ^{n-k},$$
with $\overbar W \leq \displaystyle\inf_{\tilde B_{2R}(p,1)} \mathcal{W}_{y}(\xi^{-1}),$ then there exists $q \in \tilde B_R(p,1)$ such that $(M, g(t), q)_{[- 2\delta^{-2},0]}$ is $(k , \delta^2)$--self--similar at scale 1. 
\end{theorem}

\begin{proof}We proceed by induction. For $k=0$ the result follows from Lemma \ref{s-s/e-d}. Assume it is true for $k$ and that we have 
$$\vol_{g(-\gamma^2)} \tilde B_{R, \gamma}(\mathcal P_{1, \xi, R}(p)) \geq \epsilon \gamma^{n-k-1}.$$ 
Since $\gamma < 1$ the induction assumption implies that there exists $q \in \tilde B_R(p,1)$ such that the pointed flow $(M, g(t), q)_{[-2\delta'^{-2},0]}$ is $(k, \delta'^2)$--self--similar at scale 1, where $\delta'$ an auxiliary constant to be fixed later. Assume to the contrary that $(M,g(t),q)_{[-2\delta'^{-2},0]}$ is not $(k+1, \delta^2)$--self--similar at scale 1. The latter implies that if $\delta'< \min( {\gamma}, \frac{1}R, \delta)$ then the Gromov--Hausdorff distance 
$$d_{GH}(\mathcal{L}_{q,1}\cap \tilde B_{2R}(q,1), (K \times \mathbb R ^k)\cap B(\tilde q, 2R) )< \delta',$$ where $\mathcal{L}_{q,1}\cap \tilde B_{2R}(q,1)$ is endowed with the metric $d_{g(-\gamma^2)}(\cdot,\cdot)$, $K \times \mathbb R^k$ is endowed with a product metric and  $\text{diam} (K) \leq A(n)\gamma.$ 
    
    Then a standard Vitali covering allows us to cover $(K \times \mathbb R ^k)\cap B(\tilde q, 2R)$ by $\{B(\tilde x_i, R\gamma)\}_{i=1}^N$ and $N \leq C(k) \gamma^{-k}$, provided that $R > 10 A(n).$ Thus we can cover  $\mathcal{L}_{q,1}\cap \tilde B_{2R}(q,1)$ by $\{\tilde B_{2R}(x_i, \gamma)\}_{i=1}^N$ hence we have that 
$$\vol_{g(-\gamma^2)}\tilde B_{2R, \gamma}(\mathcal{L}_{q,1}\cap \tilde B_{2R}(q,1)) \leq \sum_{i=1}^N \vol_{g(-\gamma^2)} \tilde B_{4R}(x_i , \gamma)\leq N v_2 \gamma^n R^n,$$ 
where for the latter inequality we applied the non--inflating estimate \eqref{non-inf}, provided that $\gamma < \frac{1}{4R}$. Put together we have that 
$$\vol_{g(-\gamma^2)}\tilde B_{2R, \gamma}(\mathcal{L}_{q,1}\cap \tilde B_{2R}(q,1)) \leq C(n, C_I, \Lambda ) R^n \gamma^{n-k}.$$ 
Using the latter inequality and the initial assumption on the entropy pinching set we have that 
$$\vol_{g(-\gamma^2)}\bigg(\tilde B_{R, \gamma}(\mathcal P_{1, \xi, R}(p))\setminus \tilde B_{2R, \gamma}(\mathcal{L}_{q,1}\cap \tilde B_{2R}(q,1))\bigg) $$
$$\geq R^n \gamma^{n-k} \bigg(\dfrac{\epsilon}{R^n\gamma}-C(n, C_I, \Lambda)\bigg) >0$$
    provided that $\gamma < \dfrac{\epsilon}{2 R^nC(n, C_I, \Lambda)}.$ 
Thus there exists $\tilde z \in \tilde B_{R, \gamma}(\mathcal P_{1, \xi, R}(p))$ and $d_{g(-\gamma^2)}(\tilde z, \tilde y)\geq 2R \gamma$ for any $\tilde y \in \mathcal{L}_{q,1}\cap \tilde B_{2R}(q,1)$. Hence we can find a point $y \in \mathcal P_{1, \xi, R}(p)$ with $d_{g(-\gamma^2)}(y, \tilde y) > R \gamma$ for any $\tilde y \in \mathcal{L}_{q,1}\cap \tilde B_{2R}(q,1)$. Then from Lemma \ref{s-s/e-d} there exists $z \in \tilde B_{2R}(q,1)$ such that $(M, g(t), z)$ is $(0, \delta'^2)$--self--similar at scale 1 and since $\delta'< \gamma$ we also have that $d_{g(-\gamma^2)}(z, y) \leq 2 D \gamma$. Thus using the triangle inequality and if $R > 2 D$ we get that $d_{g(-\gamma^2)}(z, \mathcal{L}_{q,1}\cap \tilde B_{2R}(q,1)) > \frac{R}{2}\gamma.$ 

The latter cannot happen provided that $\delta' \leq \delta'(n, C_I, \Lambda, R,\gamma),$ from \cite[Corollary 4.1]{PanGian}. To see this we can argue by contradiction and take $\delta'_j \to 0$ and using Theorem \ref{compactness}\eqref{II} we can derive the same shrinking soliton limit for $(M_j, g_j(t), z_j)_{[-2\delta'^{-2}_j, 0]}$ and $(M_j, g_j(t), q_j)_{[-2\delta'^{-2}_j,0]}$ which implies that $z_j$ is arbitrarily close to $\mathcal{L}_{q_j,1}\cap \tilde B_{2R}(q_j,1)$ thus leading to a contradiction.
\end{proof}

\begin{COROLLARY}\textup{(quantitative splitting at scale $r$)} Let $(M, g(t), p)_{[-2\xi^{-1}, 0]} \in \mathcal{F}(n,C_I, \Lambda)$ and $r \in (0,1)$, $\epsilon, \delta >0.$ Then if $R \geq R (n, C_I, \Lambda)$, $\gamma \leq \gamma(n, C_I, \Lambda, R, \epsilon)$, $\xi \leq \xi(n, C_I, \Lambda, R ,\delta, \gamma, \epsilon)$ and 
$$\vol_{g(-r^2 \gamma^2)}\tilde B_{R, r \gamma}(\mathcal P_{r, \xi, R}(p)) \geq \epsilon \gamma ^{n-k} r^n,$$
where $\overbar W \leq \displaystyle \inf_{\tilde B_{2R}(p,r)}\mathcal{W}(r \xi^{-1})$ there exists $q \in \tilde B_R(p,r)$ such that $(M, g(t), q)_{[-2\delta^{-2},0]}$ is $(k, \delta^2)$--self--similar at scale $r$.

\end{COROLLARY}

\begin{proof} The proof is done by scaling the metric $\tilde g(t)= r^{-2}g(r^2t)$. Since $\tilde B_R^{g}(p,r)=\tilde B_R^{\tilde g}(p,1)$, from Lemma \ref{properties of D_R}\eqref{3}, the scale invariance and the monotonicity of the pointed entropy functional we conclude that $\overbar W \leq \inf_{\tilde B^{\tilde g}_{2R}(p,1)} \mathcal{W}^{\tilde g}_y(\xi^{-1})$ thus for $\mathcal{P}^{\tilde g}_{1,\xi, R}(p)$ with respect to the metric $\tilde g$ and fixed reference value $\overbar W$ we have
$$\vol_{\tilde g (-\gamma^2)}\tilde B^{\tilde g}_{R, \gamma}(\mathcal P^{\tilde g}_{1, \xi, R}(p))\geq \epsilon \gamma ^{n-k},$$
   and the result follows from Theorem $\ref{q.split}$.
\end{proof}

\noindent
\inlinesubsection{$\eta$--regularity} For a flow $(M, g(t))_{[-T,0]} \in \mathcal{F}(n,C_I, \Lambda)$ we define the curvature radius at a point $x \in M$ 
$$ r_{Rm}^R(x)= \sup\bigg\{r \leq \sqrt{T} \mid \:|Rm|_{g(t)} \leq r^{-2} , \text{ for any } (x,t) \in \tilde B_R(x,r)\times [-r^2,0] \bigg\}.$$ 

\noindent

Since the flow is in the class $\mathcal{F}(n,C_I, \Lambda)$ the set over which we take the supremum is non--empty and the curvature radius is well defined. We mention the following elementary property which is based on how the curvature tensor behaves under scaling. If $\tilde g(t) = \lambda^{-2}g(\lambda^2 t)$ then 

\begin{equation}
\label{scaling of curvature radius}
      r_{Rm}^{R, \tilde g}(x) = \lambda^{-1} r_{Rm}^{R,g}(x) 
\end{equation}

We prove a quantitative version of $\eta$--regularity, which roughly says that if the pointed flow $(M, g(t), x)_{[-2\eta^{-1},0]}$ is sufficiently $(n, \eta)$--self--similar at some scale then we have uniform curvature bounds in an entire parabolic ball based at $x$ and up to time $t=0.$ The main issue to solve is that $(n, \eta)$--self--similarity gives uniform curvature bounds up to time $t = -\eta$, thus we need to extend these curvature bounds up to $t=0.$ To do that we need Perelman's pseudolocality theorem, see \cite{Per_02} or \cite[Theorem 21.2]{Chow III} which we recall here (written suitably for our purposes).

\begin{theorem}\textup{(pseudolocality theorem)}\label{pseudo} Let $(M, g(t))_{[-T, 0]}$ be a closed Ricci flow such that at some time $t_0$ we have the curvature bound $|Rm|_{g(-t_0^2)}\leq \frac{1}{|t_0|^2r_0^2}$ in $\tilde B_{r_0}(p, |t_0|)$. Then there exist $\epsilon_0, \delta_0>0$ depending on $n$ such that for $\epsilon \leq \epsilon_0$, $\delta\leq \delta_0$ if the curvature is bounded in $[-t_0^2, -t_0^2 + (\epsilon r_0)^2]$ and $\vol_{g(-t_0^2)}\tilde B_{r_0}(p, |t_0|) > (1-\delta)\omega_n r_0^n |t_0|^n$ then 
$$|Rm|_{g(t)} \leq \frac{1}{\epsilon|t_0|^2r_0^2},$$
    for any $(x,t) \in \tilde B_{\epsilon r_0}(p, |t_0|)\times [-t_0^2, -t_0^2 + t_0^2 (\epsilon r_0)^2].$
\end{theorem}
\noindent
We can now prove the following:

\begin{theorem}\textup{($\eta$--regularity)}\label{eta} Let $(M, g(t))_{[-2\eta^{-1}, 0]}\in \mathcal{F}(n,C_I, \Lambda)$ and $\eta'\in (0,1).$ If the pointed flow $(M, g(t), x)_{[-2\eta^{-1},0]}$ is $(n-1, \eta)$--self--similar at scale 1 then for $\eta \leq \eta(n, C_I, R, \eta')$ we have that $ r_{Rm}^R(x)> \eta'^{-1}.$ 
    
\end{theorem}

\begin{proof} Since $(M, g(t), x)_{[-2\eta^{-1},0]}$ is $(n-1, \eta)$--self--similar at scale 1 there exists a model flow of the form $(\mathbb R ^{n-1}\times M', g_s(t), \tilde x)_{(-\infty,0)}$. The only such $(n-1)$--self--similar flow is $\mathbb R^n$ thus $(M, g(t), x)_{[-2\eta^{-1},0]}$  is $(n, \eta)$--self--similar at scale 1 and the model flow is the Gaussian soliton. 

We have that $|Rm|_{g(t)}\leq \eta$ in $\tilde B_{\eta^{-1}}(x, 1)\times [-\eta^{-1},-\eta]$. If $y \in \tilde B_R(x,  \eta'^{-1})$ then the type--I bound implies that 
$$d_{g(-1)}(y,x) \leq (\eta'^{-1})^{2C_I} R  \eta'^{-1} < \eta^{-1},$$
provided that $\eta < \dfrac{\eta'^{2C_I +1}}{R}.$ 

In particular, we conclude that if $\eta \leq \eta(C_I, R, \eta')$ then $|Rm|_{g(t)} \leq \eta'$ in $\tilde B_R(x, \eta'^{-1})\times [-\eta'^{-2}, -\eta'^2].$ However, we want to have the curvature bound up to $t=0.$ 

We argue by contradiction and assume that no matter how small $\eta$ is we have that $ r_{Rm}^R(x) \leq \eta'^{-1}$ which means that there exists $y \in \tilde B_R(x, \eta'^{-1}) $ and $s_0 \in [-\eta'^{-2},0]$ such that $|Rm|_{g(s_0)} (y) > \eta'.$ We may assume that $-\eta< s_0$ since else there is nothing to prove by the above estimates. 

Consider  $z \in \tilde B_{r_0}(y, |t_0|)$ for $t_0 = - \eta'^{-1}$, $r_0 = \epsilon^{-1}$, where $\epsilon \leq \epsilon(n)$ is given by Theorem \ref{pseudo}. Then if we take $\eta< \min \{\frac{\eta'^{2C_I +1}\epsilon}{2}, \eta'^{2C_I +1}\frac{1}{2R}, \eta'^2 \epsilon \} $ Lemma \ref{properties}\eqref{4} implies that $z \in \tilde B_{\eta^{-1}} (x, 1)$ thus we have that $|Rm|_{g(-t_0^2)}< \eta < \eta'^2\epsilon^2 = |t_0|^{-2}r_0^{-2}$ in $\tilde B_{r_0}(y, |t_0|)$. The almost Euclidean volume estimate at $-t_0^2$ follows from $(n, \eta)$--self--similarity and also the curvature is bounded in $[-\eta'^{-2}, -\eta'^{-2} + \eta'^{-2}\epsilon^2r_0^2] = [-\eta'^{-2}, 0]$ since the flow is smooth up to $t=0.$ Thus from Theorem \ref{pseudo} we get that 
$$|Rm|_{g(t)} \leq \frac{1}{\epsilon^2|t_0|^2 r_0^2} = \eta'^2 < \eta',$$
in $\tilde B_{\epsilon r_0} (y, \eta'^{-1})\times [-\eta'^{-2},0]$ which contradicts our assumption at the point $(y, s_0).$ 
\end{proof}

\section{The decomposition theorem}  \label{section4}

In this section we prove the main decomposition theorem. Starting with an arbitrary ball at scale $1$ we show that it can be covered by balls that satisfy uniform curvature bounds. 
Moreover, we prove that the balls have bounded $(n-2)$--content. In particular we have the following: 

\begin{theorem}\label{b-ball dec} Let $(M, g(t), p)_{[-2,0]} \in \mathcal{F}(n,C_I, \Lambda)$, $M$ is orientable and $R \geq R(n, C_I, \Lambda).$ Then 
$$\tilde B_R(p,1) \subset \bigcup_b \tilde B_R(x_b, r_b), $$
  where $ r^{2R}_{Rm}(x_b) > 2r_b,$ with $r_b < 1$ and 
  $$\sum_b r_b^{n-2}\leq C(n, C_I, \Lambda, R).$$
\end{theorem}

In order to prove the latter we will need to decompose the scale distance ball into balls of several types based on the entropy pinching set and Theorem \ref{q.split}. We set up the notation 

\begin{itemize}
    \item[1)] $\tilde B_R (x_b, r_b)$ will denote a ball such that $ r _{Rm}^{2R}(x_b) > 2 r_b.$ 
    \item[2)] $\tilde B_R (x_c, r_c)$ will denote a ball such that $ r_{Rm}^{2R}(x_c)\leq 2r_c$ and $$\vol_{g(-\gamma^2r_c^2)}\tilde B_{4R, \gamma r_c}(\mathcal P_{r_c, \xi, 4R}(p)) \geq \epsilon \gamma^2 r_c^n.$$
    \item[3)]  $\tilde B_R(x_d, r_d) $ will denote a ball such that $ r_{Rm}^{2R}(x_d)\leq 2r_d$, $\mathcal{P}_{r_d,\xi, R}(x_d)\neq \emptyset$ and $$\vol_{g(-\gamma^2r_d^2)}\tilde B_{4R, \gamma r_d}(\mathcal P_{r_d, \xi, 4R}(p)) < \epsilon \gamma^2 r_d^n.$$ 
    \item[4)] $\tilde B_R(x_{\tilde d}, r_{\tilde d})$ will denote a ball such that $ r_{Rm}^{2R}(x_{\tilde d})\leq 2r_{\tilde d}$
    $$\vol_{g(-\gamma^2r_{\tilde d}^2)}\tilde B_{4R, \gamma r_{\tilde d}}(\mathcal P_{r_{\tilde d}, \xi, 4R}(p)) < \epsilon \gamma^2 r_{\tilde d}^n.$$ 
    \item[5)] $\tilde B_R (x_e, r_e)$ will denote a ball such that $ r_{Rm}^{2R}(x_e)\leq 2r_e$ and $\mathcal P_{r_e, \xi, 4R}(x_e) = \emptyset.$ 
    \item[6)] $\tilde B_R (x_f, r_f)$ will denote a ball for which we know nothing about. 
\end{itemize}

\noindent
\inlinesubsection{Decomposition of d--balls} If we start with a $d$--ball then we have insufficient information about the symmetries of the flow at that scale. However, we can start by decomposing it into balls of different types. We can then recursively decompose the remaining $d$--balls until none are left. The smoothness of the flow together with the $(n-2)$--content estimate of the initial pinching set gives that the process terminates after finitely many steps and that the $c$--balls left on the covering have small $(n-2)$--content. 

\begin{LEMMA} \label{d-ball dec} Let $(M, g(t),p)_{[-2\xi^{-1},0]}\in \mathcal{F}(n,C_I, \Lambda)$. If $R \geq R(n, C_I, \Lambda),$ $\gamma < \frac{1}{10}$, $\epsilon \leq \epsilon(n, C_I, \Lambda, R)$ and $\xi \leq \xi(n, C_I, \Lambda , R, \gamma, \epsilon)$ such that $ r_{Rm}^{2R}(p) \leq 2$ , $\mathcal{P}_{1,\xi, 4R}(p) \neq \emptyset$ and
\begin{equation}
\label{labeled}
\vol_{g(-\gamma^2)}\tilde B_{4R, \gamma}(\mathcal P_{1, \xi, 4R}(p))<  \epsilon \gamma^2, 
\end{equation}
where $\overbar W \leq \inf_{\tilde B_{2R}(p,1)} \mathcal{W}_y(\xi^{-1})$, then 
$$\tilde B_R (p,1) \subset \bigcup_b \tilde B_R (x_b, r_b)\cup \bigcup_c \tilde B_R(x_c, r_c) \cup \bigcup _e \tilde B_R (x_e, r_e),$$ 
with estimates $$\sum_b r_b^{n-2} + \sum_e r_e^{n-2} \leq C(n, C_I, \Lambda, \gamma, R),$$
$$\sum_c r_c^{n-2} \leq C(n,C_I, \Lambda, R)\epsilon.$$
    
\end{LEMMA}

\begin{proof} We start by constructing a covering on $\tilde B_R(p,1)$ in the metric space $(M, d_{g(-\gamma^2)})$. Choose a maximal set of points $x_f \in \tilde B_R(p,1)$ so that $\tilde B_{R/2}(x_f, \gamma)$ are pairwise disjoint. Then $f$ must be finite since else, by compactness of $M$, $x_f \xrightarrow[d_{g(-\gamma^2)}]{} x_0$ and that would contradict the disjoint property. Furthermore, by maximality, the balls $\{\tilde B_R(x_f, \gamma)\}_f$ cover $\tilde B_R(p,1).$ By separating the balls into categories we have that
$$\tilde B_R (p,1) \subset \bigcup_b \tilde B_R (x_b, \gamma) \cup \bigcup_c \tilde B_R (x_c, \gamma) \cup \bigcup_d \tilde B_R (x_d, \gamma)\cup \bigcup_e \tilde B_R (x_e, r_e).$$ 
We first estimate the $(n-2)$--content of $b$--balls. Let $y \in \tilde B_{R/2} (x_b, \gamma)$ then triangle inequality and distance distortion \eqref{weak dist dist} imply that $d_{g(-1)}(y,p) < R(2+\gamma)$ thus 
$\bigcup_b \tilde B_{R/2}(x_b, \gamma) \subset \tilde B_{R(2+\gamma)}(p,1).$ Using the non--collapsing estimate \eqref{non-coll}, the non--inflating estimate \eqref{non-inf} and  \eqref{volume comp}, we have that 
$$N_b \gamma^n v_1 \leq\sum_{b=1}^{N_b} \vol_{g(-\gamma^2)} \tilde B_{R/2}(x_b, \gamma) \leq \vol_{g(-\gamma^2)}\tilde B_{R(2+\gamma)}(p,1) $$
$$\leq C(n) \vol_{g(-1)}\tilde B_{R(2+\gamma)}(p,1) \leq C(n) R^n (2+\gamma)^n v_2,$$
thus in total $N_b \leq C(n) R^n \frac{(2+\gamma)^n}{\gamma^n}\frac{v_2}{v_1}= C(n, C_I, \Lambda, R, \gamma)$ which implies that 
$$\sum_{b=1}^{N_b} \gamma^{n-2}\leq C(n, C_I, \Lambda, R, \gamma).$$
In exactly the same way we estimate the $(n-2)$--content of $e$--balls. To estimate the $(n-2)$--content of $c$--balls and $d$--balls we argue differently. We first estimate that of $d$--balls.  For each $x_d$ we find $y_d$ such that $\tilde B_{2R}(y_d, \gamma) \subset \tilde B_{4R, \gamma}(\mathcal P_{1, \xi, 4R}(p)).$ Indeed, let $y\in \mathcal P_{\gamma,\xi, 4R}(x_d)$ thus $y \in \tilde B_{2R}(x_d, \gamma)$ and $\mathcal W_y(\xi\gamma^2)- \overbar W < \xi$ (such a point $y$ exists since $\mathcal P_{\gamma, \xi, 4R}(x_d) \neq \emptyset).$ From the monotonicity of the pointed entropy we have that $\mathcal W_y(\xi)- \overbar W < \xi$.  Furthermore,  triangle inequality and \eqref{weak dist dist} imply that 
$d_{g(-1)}(y,p) \leq R + K + 2R \gamma$ and for $R > 2K$ we have that 
$d_{g(-1)}(y,p)< 2R$ thus in total $y \in \mathcal P_{1, \xi, 4R }(p).$ If we set $y_d=y$ we have proven the desired inclusion.

Define $I_d = \{l \in \{1, \dots N_d\} \mid \tilde B_{2R}(y_d, \gamma) \cap \tilde B_{2R}(y_l, \gamma) \neq \emptyset \}.$ Let $l \in I_d$ then $d_{g(-\gamma^2)}(y_d, y_l) < 4\gamma R$ thus $d_{g(-\gamma^2)}(x_d, x_l)\leq 8 \gamma R$ and if $x \in \tilde B_{R/2}(x_l,\gamma)$ then $d_{g(-\gamma^2)}(x, x_d)< 9 \gamma R$ hence \linebreak 
$\bigcup_{l\in I_d}\tilde B_{R/2}(x_l, \gamma) \subset \tilde B_{9R}(x_d, \gamma)$ thus from the non--collapsing estimate \eqref{non-coll} and the non--inflating estimate \eqref{non-inf} we have that 
     $$\sharp I_d \leq 9^n R^n \frac{v_2}{v_1} = C'(n, C_I, \Lambda, R).$$ 
We conclude that a point can belong to at most $\tilde C(n, C_I, \Lambda, R) $ many balls $\tilde B_{2R}(y_i, \gamma) $ thus we have that 
$$\sum_{d=1}^{N_d}\vol_{g(-\gamma^2)} \tilde B_{2R}(y_d, \gamma) \leq \tilde C(n, C_I, \Lambda, R)\vol_{g(-\gamma^2)}\bigg(\bigcup_{d} \tilde B_{2R}(y_d, \gamma)\bigg)\leq $$
$$\leq \tilde C(n, C_I, \Lambda, R)\vol_{g(-\gamma^2)}\tilde B_{4R, \gamma} (\mathcal P_{1, \xi, 4R}(p)).$$
Using \eqref{non-coll} again and \eqref{labeled} we obtain that
$N_d \gamma^n v_1 \leq \tilde C(n, C_I, \Lambda, R)\epsilon \gamma^2$ thus 
$$\sum_d \gamma^{n-2} \leq C(n, C_I, \Lambda, R) \epsilon,$$
where $C(n, C_I, \Lambda, R) = \tilde C(n, C_I, \Lambda, R) v_1^{-1}.$ 
Since we only used the fact that $\mathcal P_{\gamma, \xi, 4R}(x_d)\neq \emptyset$ for every $d$ we can use the same reasoning to estimate the $(n-2)$--content of $c$--balls. This finishes the first step of the covering. 

We are left with $d$--balls which we have to refine until there are none left by repeating the same argument. In what follows the constants $C(n, C_I, \Lambda, \gamma, R)$ and $C(n, C_I, \Lambda, R)$ will denote the constants from step 1. We proceed by induction and assume we have:

$$\tilde B_R(p,1 )\subset \bigcup_{d=1}^{N_d^{(i)}} \tilde B_R(x_d^{(i)}, \gamma^i)\cup \bigcup_{j=1}^{i}\bigg (\bigcup_{b=1}^{N_b^{(j)}}\tilde B_R(x_b^{(j)}, \gamma^j)\cup \bigcup_{c=1}^{N_c^{(j)}}\tilde B_R(x_c^{(j)}, \gamma^j)\cup \bigcup_{e=1}^{N_e^{(j)}}\tilde B_R (x_e^{(j)}, \gamma^j)\bigg)$$
with estimates 
$$\sum_{j=1}^i \bigg( \sum_{b=1}^{N_b^{(j)}}\gamma^{j(n-2)} + \sum_{e=1}^{N_e^{(j)}}\gamma^{j(n-2)} \bigg) \leq C(n, C_I, \Lambda, R, \gamma) \sum_{j=0}^{i-1}(C(n, C_I, \Lambda, R)\epsilon)^{j},$$ 
$$\sum_{j=1}^{i}\sum_{c=1}^{N_c^{(j)}}\gamma^{j(n-2)} \leq \sum_{j=1}^i (C(n, C_I, \Lambda, R)\epsilon)^j,$$ 
$$\sum_{d=1}^{N_d^{(i)}}\gamma^{i(n-2)}\leq (C(n, C_I, \Lambda, R)\epsilon)^i.$$ 
To each $d$--ball at scale $\gamma^i$ we repeat the covering of step 1 at that scale, thus we find a maximal cover by balls $\tilde B_R(x_f, \gamma^{i+1})$ such that $\tilde B_{R/2}(x_f, \gamma^{i+1})$ are pairwise disjoint and we separate these balls into categories. We first estimate $N_b^{(i+1)}$ which is the number of $b$--balls we get at the $(i+1)$--step from each $d$--ball. As in step 1, 
$\displaystyle\bigcup_{b=1}^{N_b^{(i+1)}}\tilde B_{R/2}(x_b^{(i+1)}, \gamma^{i+1}) \subset \tilde B_{R(2+\gamma)}(x_d^{(i)}, \gamma^i)$ so from \eqref{non-coll}, \eqref{volume comp} and \eqref{non-inf} we have that 
$$\gamma^{ni+n}v_1 N_b^{(i+1)} \leq C(n) R^n (2+\gamma)^n v_2 \gamma^{ni},$$ 
which implies that $N_b^{(i+1)} \leq C(n, C_I, \Lambda, R, \gamma)$. Thus $$\sum_{b=1}^{N_b^{(i+1)}}\gamma^{(i+1)(n-2)} \leq C(n, C_I, \Lambda, R,  \gamma)\gamma^{i(n-2)}$$ and since we do this to $N_d^{(i)}$--many balls summing over them and using the inductive estimate we have that the $(n-2)$--content of the new $b$--balls is bounded by $ C(n, C_I, \Lambda, R,  \gamma) (C(n, C_I, \Lambda, R)\epsilon)^{i}$. Put together after $(i+1)$--steps we have
$$\sum_{j=1}^{i+1} \sum_{b=1}^{N_b^{(j)}} \gamma^{j(n-2)} \leq C(n, C_I, \Lambda, R, \gamma) \sum_{j=0}^i (C(n, C_I, \Lambda, R)\epsilon)^j.$$
With the exact same reasoning we estimate the $(n-2)$--content of $e$--balls. 
In order to estimate the content of $c$--balls we repeat the argument of step 1. So we can find $y_c^{(i+1)}$ so that $\tilde B_{2R}(y_c^{(i+1)}, \gamma^{i+1})\subset \tilde B_{4R, \gamma^{i+1}}(\mathcal P_{\gamma^{i}, \xi, 4R}(x_d^{(i)})).$ 
The set $I_c= \{ l \in\{1, \dots, N_c^{(i+1)}\}\mid \tilde B_{2R}(y_c^{(i+1)}, \gamma^{i+1})\cap \tilde B_{2R}(y_l ^{(i+1)}, \gamma^{i+1})\neq \emptyset\}$ satisfies $v_1 \gamma^{ni+n} \sharp I_c \leq 9^n R^n \gamma^{ni+n}v_2$ hence we get that 
$$\sum_{c=1}^{N_c^{(i+1)}} \gamma^{(i+1)(n-2)}\leq C(n, C_I, \Lambda, R)\epsilon \gamma^{(n-2)i}.$$
Since we do this to $N_d^{(i)}$--many balls the content of the new $c$--balls is bounded by $(C(n, C_I, \Lambda, R)\epsilon)^{i+1}$. Put together we have the content estimate for $c$--balls 
$$\sum_{j=1}^{i+1}\sum_{c=1}^{N_c^{(j)}} \gamma^{j(n-2)} \leq \sum_{j=1}^{i+1}(C(n, C_I, \Lambda, R)\epsilon)^j.$$ 
The exact same reasoning is used for the content of $d$--balls at the $(i+1)$--step, thus
$$\sum_{d=1}^{N_d^{(i+1)}} \gamma^{(n-2)(i+1)}\leq (C(n, C_I, \Lambda, R)\epsilon)^{i+1}.$$ 
This completes the induction. 

We now claim that after finitely many steps there will be no $d$--balls left. Assume that we can take $i \to \infty$ and $x_i \in \mathcal P_{\gamma^{i}, \xi, 4R}(x_d^{(i)})$. Then for $ g_i(t) = \gamma^{-2i}g(\gamma^{2i}t)$ since the flow is smooth we have that 
$$(M,  g_i(t), x_i) \to (\mathbb R^n, g_{Eucl.}(t), x_\infty),$$
and $x_\infty \in S_{point}$ since every point of the Gaussian soliton belongs to the point spine. Then from the conjugate heat kernel bounds of Lemma \ref{properties}\eqref{2}, combined with \cite[Proposition 4.3]{PanGian}, we have that the conjugate heat flows $\nu_{(x_i,0)} \to \nu_\infty \in \mathcal S$. From Theorem \ref{compactness}\eqref{III} we get that $\mathcal W_{x_i}^{g_i}(|t|) \to \mu(g_{Eucl})=0$, where $\mathcal W_{x_i}^{g_i}$ is used to denote the pointed entropy with respect to the metric $g_i.$ Since $\mathcal W_{x_i}^{g_i}(|t|) =\mathcal  W_{x_i}(\gamma^{2i}|t|) $ we conclude that 
$$\overbar W + \xi > \mathcal W_{x_i}(\gamma^{2i}\xi) = \mathcal W_{x_i}^{g_i}(\xi) > -\epsilon',$$ 
for any $\epsilon' >0$. Thus $\overbar W + \xi \geq 0$ and since from Perelman's Harnack inequality we have that $\mathcal W_x(|t|)\leq 0$ at any scale and any point we conclude that $\mathcal W_x(\xi) < \overbar W +\xi  $ for any $x \in \tilde B_{2R}(p,1)$ so $\tilde B_{2R}(p,1) \subset \mathcal P_{1,\xi, 4R}(p)$. From the non--collapsing estimate \eqref{non-coll} and the initial assumption \eqref{labeled} we conclude that $v_1 \leq \epsilon \gamma^2 < \epsilon$ and the latter is a contradiction if $\epsilon < \frac{v_1}{2}.$ 

In conclusion, after finitely many steps there will be no $d$--balls left. If $\epsilon < \frac{1}{2C(n, C_I, \Lambda, R)}$ we have that the series on the estimates converge which concludes the proof.  
\end{proof}

\begin{REMARK} If we have instead an initial ball $\tilde B_R(x, r)$ that satisfies the volume estimate  $\vol_{g(-\gamma^2r^2)}\tilde B_{4R, r\gamma}(\mathcal P_{r,\xi, 4R}(x)) < \epsilon \gamma^2 r^n$ with $\overbar W \leq \displaystyle \inf_{\tilde B_{2R}(x,r)}\mathcal{W}_y(r^2\xi^{-1})$ then rescaling the metric to $\tilde g$ we get the set $\mathcal{P}^{\tilde g}_{1, \xi, R}(p)$ with $\overbar W \leq \inf_{\tilde B_{2R}^{\tilde g}(p,1)}\mathcal{W}^{\tilde g}_y(\xi^{-1})$ that satisfies \eqref{labeled}. Thus we may apply Lemma \ref{d-ball dec} and after scaling back get a covering for $\tilde B_R (x,r)$ together with the rescaled estimates. 
    
\end{REMARK}

\noindent
\inlinesubsection{Necks of maximal symmetry on type--I Ricci flows} We introduce the notion of an $(n-2, \delta)$--neck, similar to the one in \cite{JiangNaber}, \cite{CJN}, see also \cite{PanGian2}. Recall that an $(n-2)$--self--similar flow is of the form $\mathbb R ^{n-2}\times M$, where $M$ is a gradient shrinking surface. If we further assume that $M$ is orientable then from \cite{Ham} we have that the flow is either $\mathbb R ^{n}$ or $\mathbb R ^{n-2}\times \mathbb S^2$. Note that in any case the point spine is the entire flow. With this in mind we give the following:

\begin{DEFINITION} \textup{(necks of maximal symmetry)} \label{def:necks_ms} Let $(M, g(t), p)_{[-2\delta^{-3}s^2, 0]} \in \mathcal F(n, C_I, \Lambda)$ and fix $R>0$ and $\tau = 10^{-10C_I-10n}$. A closed set $\mathcal{C} \subset \tilde B_{2R}(p,s)$ with $p \in \mathcal{C}$ together with a radius function $r: \mathcal{C} \to \mathbb R_{>0}$ is called an $(n-2, \delta)$--neck at scale $s$ in $\tilde B_{2R}(p,s)$ if for any $x \in \mathcal{C}$ we have the following properties 

\begin{itemize}
    \item [n1)] The balls $\tilde B_{\tau^2 R}(x, r_x)$ are pairwise disjoint.
    \item[n2)]\label{n2} $\mathcal W_x(\delta r_x^2) - \mathcal W_x(\delta^{-1}s^2)< \delta.$ 
    \item[n3)]For any $r \in [r_x, \delta^{-1}s]$ the pointed flow $(M, g(t), x)_{[-2\delta^{-2}r^2,0]}$ is $(n-2, \delta^2)$--self--similar at scale $r$ with $\mathbb R ^{n-2}\times \mathbb S^2$ a model flow.  
    \item[n4)] For any $r \in [r_x, s]$ such that $\tilde B_{2R}(x,r) \subset \tilde B_{2R}(p,s)$ we have that $\tilde B_{R}(x,r)\subset \tilde B_{\tau R, r}(\mathcal{C}).$
    \item[n5)] $r_y \leq \sigma r_x + \delta D_{2R}(x,y)$ for any $x, y \in \mathcal{C}$ and some $\sigma \in (1,2).$
\end{itemize}
    
\end{DEFINITION}
\begin{REMARK}\begin{enumerate}
    \item [1)] The radius function is assumed to be strictly positive since the flow is smooth at $t=0$. In fact, if $r_{x_i}\to 0$ as $i \to \infty$ for some $x_i \in \mathcal{C}$ then rescaling the flow would contradict the smoothness up to $t=0$, due to $n3)$. This, together with the property $n1)$, implies that $\mathcal{C}$ is indeed a finite set. To see the latter, if $\mathcal{C}$ was infinite then let a sequence $x_i$. Since $M$ is compact $x_i$ has, up to a subsequence, a limit. So $x_i$ is a Cauchy sequence and $\tilde B_{\tau^2 R }(x_i, r_0)\cap \tilde B_{\tau^2R}({x_j, r_0)}\neq \emptyset$, where $r_0 = \inf_{\mathcal{C}} r_x>0$ and this contradicts property $n1).$
    
\item [2)] 
Condition $n5)$ will be called $2R$--scale $(\sigma, \delta)$--Lipschitz.
\item[3)] A closed set $\mathcal{C} \subset U$, where $U$ an open subset of $M$, will be called a weak $(n-2, \delta)$--neck in $U$ if it satisfies conditions $n1)-n3)$. 
\item[4)] An $(n-2, \delta)$--neck $(\mathcal C, r_x)$ is also an $(n-2, \delta, \eta)$--neck in the sense of \cite{PanGian2}. Clearly properties $n1), n2), n5)$ are the same. From property $n3)$ we have that for any $x \in \mathcal{C}$ if $\eta\leq \eta(n)$ and $\delta\leq \delta(\eta)$ the flow $(M, g(t), x)_{[-2\eta^{-1},0]}$ is not $(n,\eta)$--self--similar at scale $r \in[r_x,\delta^{-1}s].$  Thus property $n3)$ is the same. Furthermore, since for the model flow the point spine is the entire flow it follows that if $\delta \leq \delta(R)$ then $\mathcal L_{x,r}\cap \tilde B_{R}(x,r) = \tilde B_{R}(x,r)$ for any $x \in \mathcal C$ thus property $n4)$ is also the same as in \cite{PanGian2}.

\end{enumerate} 
    
\end{REMARK}

\noindent
Given an $(n-2, \delta)$--neck at scale 1 we associate to it a packing measure $\displaystyle\mu = R^{n-2} \sum_{x \in \mathcal{C}} r_x^{n-2}\delta_x$, where $\delta_x$ is the Dirac measure at the point $x.$ We recall the main result proven in \cite{PanGian2}.

\begin{theorem}\textup{(neck structure theorem)}\label{neck structure} Let $(M, g(t), p)_{[-2\delta^{-3},0]} \in \mathcal F (n, C_I, \Lambda)$ and $(\mathcal{C}, r_x)$ an $(n-2, \delta)$--neck at scale 1 in $\tilde B_{2R}(p,1)$. Then if $\delta \leq \delta(n, C_I, \Lambda, R)$ for any $x \in \mathcal{C}$ and $r\in [r_x, 1]$ such that $\tilde B_{2R}(x,r)\subset \tilde B_{2R}(p,1)$ we have that 
$$A^{-1} R^{n-2}r^{n-2} \leq \mu(\tilde B_R(x,r))\leq A R^{n-2}r^{n-2},$$ 
    where $A$ a constant that depends on $n$ and $C_I.$ 
\end{theorem}

\noindent
We now prove the following useful lemmas.     

\begin{LEMMA}\textup{(higher order symmetries)}\label{higher order sym} Let $\gamma, \delta' >0$ and $(M, g(t), x)_{[-2\xi''^{-1}r^2,0]} \in \mathcal F(n, C_I , \Lambda)$ be a $(0,\xi'')$--self--similar flow at scale $r \in [\gamma, \delta'^{-1}]$ and $(M, g(t),p)_{[-2\xi''^{-1},0]}$ be an $(n-2, \xi''^2)$--self--similar flow at scale 1 with model flow $\mathbb R ^{n-2}\times \mathbb S^2$ and $x \in \tilde B_{R_0}(p,\max \{1,r\}).$ Then if $\xi'' \leq \xi''(n, C_I, \Lambda, R_0,\gamma, \delta')$ we have that $(M, g(t), x)_{[-2\delta'^{-1}r^2,0]}$ is $(n-2,\delta'^2)$--self similar at scale $r$ with model flow $\mathbb R ^{n-2}\times \mathbb S^2.$ 
    
\end{LEMMA}
\begin{proof} Assume that we can find $\xi''_j \to 0$ such that $(M_j, g_j(t), x_j)_{[-2\xi''^{-1}_j r^2,0]}$ is $(0,\xi''_j)$--self--similar at scale $r$, $(M_j, g_j(t), p_j)_{[-2\xi''^{-2}_j,0]}$ is $(n-2, \xi''_j{^2})$--self--similar at scale 1 with model flow $\mathbb R ^{n-2}\times \mathbb S^2$, $x_j \in \tilde B_{R_0}(p_j, \max\{1,r\})$  and $(M_j, g_j(t), x_j)_{[-2\xi''^{-1}_j,0]}$ is not $(n-2, \delta'^2)$--self--similar at scale $r.$ 

Then if $\xi''_j \leq \xi''_j(n, C_I, \Lambda, R_0, \gamma, \delta')$ we have as $j \to \infty$, from Theorem \ref{compactness}\eqref{II}, that the pointed flows converge to the same soliton which must be $\mathbb R ^{n-2}\times \mathbb S^2$ thus arriving at a contradiction. 
\end{proof}

\begin{LEMMA}\textup{(curvature radius on necks)}\label{curv radius} If $(\mathcal{C},r_x)$ is a weak $(n-2, \delta)$--neck at scale 1 in an open set $U$, where $\tilde B_{2R}(p,1) \subset U,$ then for any $x \in \mathcal{C}$ and $\delta \leq \delta(n)$ we have that $ r_{Rm}^{2R}(x) \leq r_x.$
    
\end{LEMMA}

\begin{proof} Let $x \in \mathcal{C}$, $r_x>0$. We want to find $t_0 \in [-r^2_x, 0]$ and $x_0 \in \tilde B_{2R} (x, r_x)$ such that $|Rm|_{g(t_0)}(x_0) > \frac{1}{r^2_x.}$ From property $n3)$ there exists a map $F : B_{g_s(-r^2_x)}(\tilde x, \delta^{-2}r_x) \to M$ diffeomorphism to its image,  $B_{g_s(-r^2_x)}(\tilde x, \delta^{-2}r_x) \subset \mathbb R^{n-2}\times \mathbb S^2$ the shrinking cylinder and $||F^* g - g_s|| < \delta^2 r^2_x$  in  $B_{g_s(-r^2_x)}(\tilde x, \delta^{-2}r_x)\times [-\delta^{-2}r^2_x, -\delta^2r_x^2].$ 

Thus for the scalar curvature we have that $|\text{Scal}_{g(-\delta^2 r_x^2)}(x)| \geq \text{Scal}_{g_s(-\delta^{2}r_x^2)}(\tilde x) - \delta^2r^2_x$. Since $\tilde x \in S_{\text{point}}$ we have that $\text{Scal}_{g_s(-\delta^{2}r_x^2)}(\tilde x) = \text{Scal}_{g_s(-1)}(\tilde x)\delta^{-2}r_x^{-2}$. So we get that 
$$|\text{Scal}_{g(-\delta^2r_x^2)}(x)| > \dfrac{C(n) - \delta^4r_x^2}{\delta^2r_x^2} > \frac{C(n)}{2\delta^2 r_x^2},$$
provided that $\delta \leq \delta(n)$ and this suffices to prove the result.
\end{proof}

\begin{LEMMA}\textup{(dichotomy lemma)} \label{dicho} Let $(M, g(t), p)_{[-2\xi^{-1},0]} \in \mathcal{F}(n, C_I, \Lambda)$ and $M$ orientable. Then for any $\epsilon>0, \eta>0$, $R \geq R(n, C_I, \Lambda)$, $\gamma\leq \gamma(n, C_I, \Lambda, R, \epsilon)$, $\delta\leq \delta(n, C_I, R, \eta)$, $\xi \leq \xi(n, C_I, \Lambda, R, \gamma, \delta, \eta, \epsilon)$ such that  
$$\vol_{g(-\gamma^2)} \tilde B_{4R, \gamma}(\mathcal P_{1, \xi, 4R}(p)) \geq \epsilon \gamma^2,$$ 
where $\overbar W \leq\inf_{\tilde B_{8R}(p,1)}\mathcal W_y(\xi^{-1}),$ either $(M, g(t), p)_{[-2\delta^{-2},0]}$ is $(n-2, \delta^2)$--self--similar at scale 1 with model flow $\mathbb R ^{n-2}\times \mathbb S^2$ or the curvature radius satisfies $ r_{Rm}^{2R}(p) > \eta^{-1}.$  
    
\end{LEMMA}

\begin{proof} We may fix the constants as in Theorem \ref{q.split} to get that there exists $q \in \tilde B_{2R}(p,1)$ such that $(M, g(t), q)_{[-2\delta^{-4},0]}$ is $(n-2, \delta^4)$--self--similar at scale 1. Since $\tilde B_{4R}(q,1) \subset \mathcal{L}_{q,1}$ we conclude that $p \in \mathcal{L}_{q,1}$ thus $(M,g(t),p)_{[-2\delta^{-2},0]}$ is $(n-2,\delta^2)$--self--similar at scale 1. If the model flow is $\mathbb R ^{n-2}\times \mathbb S^2$ then we have nothing to prove. Else the flow is $(n, \delta^2)$--self--similar at scale 1 so if $\delta \leq \delta(n, C_I, R| \eta)$ from $\eta$--regularity we get that $ r_{Rm}^{2R}(p) > \eta^{-1}$ which proves the dichotomy. 
\end{proof}

We will decompose a $c$--ball $\tilde B_{2R}(p,1)$ in two steps. First in Lemma \ref{maximal neck}, using Lemma \ref{dicho} and Lemma \ref{curv radius}, we initially cover the ball by balls of different types, except $b$--balls, so that the collection of the centers form a weak neck. We then recursively apply the same covering to the remaining $c$--balls until none are left (due to the smoothness of the flow this will happen after finitely many steps). In that way we get a maximal weak neck in $\tilde B_{4R}(p,1)$ that contains only $\tilde d$--points. However, property $n5)$ and hence property $n4)$ fail to hold. 

To fix this we will redefine the radius function and the weak neck so that in Lemma \ref{neck construction/estimates} we get an $(n-2,\delta)$--neck. The latter may contain $c$--points but since the weak neck is maximal an application of the neck structure theorem, Theorem \ref{neck structure}, shows that the content of the $c$--balls is small. 

\begin{LEMMA}\textup{(weak neck construction)}\label{maximal neck} Let $(M, g(t), p)_{[-2\xi^{-1},0]} \in \mathcal{F}(n,C_I, \Lambda),$ $M$ orientable, such that $ r_{Rm}^{2R}(p) \leq 2.$ For any $\epsilon>0,$ $R \geq R(n, C_I, \Lambda)$, $\delta'\leq \delta'(n, C_I, R)$, $\gamma \leq \gamma(n, C_I, \Lambda, R, \epsilon)$, $\xi\leq \xi(n, C_I, \Lambda, R,\delta', \gamma, \epsilon)$ if 
$$\vol_{g(-\gamma^2)}\tilde B_{4R,\gamma}(\mathcal P_{1, \xi, 4R}(p))\geq \epsilon \gamma^2,$$
   where $\overbar W \leq \inf_{\tilde B_{8R}(p,1)} \mathcal{W}_y(\xi^{-1})$, there exists $\mathcal{C}\subset \tilde B_{4R}(p,1)$ such that $(\mathcal{C}, r_x)$ is a weak $(n-2, \delta')$--neck in $\tilde B_{4R}(p,1)$ and 
    \begin{itemize}
        \item [i)] $\tilde B_{2R}(p,1)\subset \tilde B_{\tau R, r_x}(\mathcal{C}).$
        \item[ii)] For any $x \in \mathcal{C}$ we have that 
        $$\vol_{g(-\gamma^2 r_x^2)}\tilde B_{4R, \gamma r_x}(\mathcal P_{r_x, \xi, 4R}(x))< \epsilon \gamma^2 r_x^n.$$
    \end{itemize}
\end{LEMMA}

\begin{proof} From Lemma \ref{dicho}, for every $R \geq R(n, C_I, \Lambda)$ and $\gamma\leq \gamma(n, C_I, \Lambda, R ,\epsilon)$, if we choose $ \delta''\leq \delta''(n,C_I, R)$ and $\xi \leq \xi(n, C_I, \Lambda, R ,\delta'', \gamma, \epsilon)$ then  $(M, g(t), p)_{[-2\delta''^{-2},0]}$ is $(n-2,\delta''^2)$--self--similar at scale 1, with model flow $\mathbb R^{n-2}\times \mathbb S^2$. Note that $\tilde B_{2R}(p,1) \subset \mathcal{L}_{p,1}$. 

Choose a maximal covering $\{\tilde B_{\tau^{3/2}R}(x_f,\gamma)\}_{f=1}^{N_f}$  of $\tilde B_{2R}(p,1)$  such that $\tilde B_{\tau^{5/3}R}(x_f, \gamma)$ are pairwise disjoint. Note that we can take $p$ to be one of the points $x_f$ and let $\mathcal{C}^{(1)} = \{x_f\}$. We first show that $\mathcal{C}^{(1)}$ is a weak $(n-2, \delta')$--neck at scale 1 in $\tilde B_{4R}(p,1)$, with $r_x^{(1)} = \gamma$ for any $x \in \mathcal{C}^{(1)}.$ Clearly $n1)$ is satisfied by construction. 

For $n2)$, since $ \mathcal{C}^{(1)} \subset \tilde B_{2R}(p,1) \cap \mathcal{L}_{p,1}$ we apply Lemma \ref{s-s/e-d} to get that if we choose $\delta''\leq \delta''(n, C_I, \Lambda, R, \gamma, \xi') $, where $\xi'>0$ another auxiliary constant to be fixed later, then $\mathcal W_x(\gamma^2 \xi')-\overbar W < \xi'$. Since $\overbar W \leq \mathcal W_x(\xi^{-1})$ and $\xi < \delta''< \xi'^{2}$ the monotonicity of the pointed entropy implies that 
$$\mathcal W_x(\gamma^2 \xi'^2) - \mathcal W_x(\xi'^{-2})< \xi'^2.$$
Choosing $\xi' < \delta'^4$ we get $n2).$ Note that the monotonicity of the pointed entropy also gives that for any $r\in [\gamma,\delta'^{-1}]$
\begin{equation}
\label{en-drop}
\mathcal W_x(r^2 \xi') - \mathcal  W_x(r^2\xi'^{-1})< \xi'.
\end{equation}

\medskip
\noindent 
For $n3),$ from \eqref{en-drop} and Lemma \ref{s-s/e-d} if $\xi'\leq \xi'(n, C_I, \Lambda, \xi'')$, where $\xi''>0$ another auxiliary constant to be fixed later,  we have that $(M, g(t), x_{(r)})_{[-2\xi''^{-2},0]}$ is $(0,\xi'')$--self--similar at scale $r$, where $x_{(r)}\in \tilde B_R(x,r)$. Then from Lemma \ref{higher order sym}, applied to the pointed flows $(M, g(t), x_{(r)})_{[-2\xi''^{-1}r^2,0]}$ and $(M, g(t), p)_{[-2\xi''^{-1},0]}$, if $\xi'' \leq \xi''(n, C_I, \Lambda, R ,\gamma, \delta')$ we get that $(M,g(t), x_{(r)})_{[-2\delta'^{-3},0]}$ is $(n-2,\delta'^3)$--self--similar at scale $r$ with model flow $\mathbb R^{n-2}\times \mathbb S^2.$  Thus for $\delta' < \frac{1}{R}$ we get that $(M, g(t), x)_{[-2\delta'^{-2},0]}$ is $(n-2, \delta'^2)$--self--similar with the same model flow. This shows property $n3).$ 

In conclusion, $(\mathcal{C}^{(1)},r_x^{(1)})$ is a weak $(n-2, \delta')$--neck at scale 1 and by construction $\tilde B_{2R}(p,1) \subset \tilde B_{\tau R, r_x}(\mathcal{C}^{(1)}).$ From Lemma \ref{curv radius} we have that $ r _{Rm}^{2R}(x) \leq 2\gamma$ for each $x \in \mathcal{C}^{(1)}.$ We separate the balls of the covering to get that 

$$\tilde B_{2R}(p,1) \subset \bigcup_c \tilde B_{\tau R}(x_c, r_c)\cup \bigcup _{\tilde d}\tilde B_{\tau R}(x_{\tilde d}, r_{\tilde d}),$$
where 
\begin{itemize}
    \item [i)]$\vol_{g(-\gamma^2 r_c^2)} \tilde B_{4R, \gamma r_c}(\mathcal P_{r_c, \xi, 4R}(x_c)) \geq \epsilon \gamma^2 r_c^n$ 
    \item [ii)] $\vol_{g(-\gamma^2 r_{\tilde d}^2)} \tilde B_{4R, \gamma r_c}(\mathcal P_{r_{\tilde d}, \xi, 4R}(x_c)) < \epsilon \gamma^2 r_{\tilde d}^n$
\end{itemize}

We want to repeat the covering on each $c$--ball until there are no $c$--balls left. We proceed by induction. Assume we have done this $i$--times and we have the following 

\begin{itemize}
    \item [i1)] We have points $x_c^{(i)} \in \tilde B_{2R}(x_c^{(i-1)}, \gamma^{i-1})$ such that $ r_{Rm}^{2R}(x_c^{(i)}) \leq 2\gamma^i,$ $$\vol_{g(-\gamma^2 \gamma^{2i})} \tilde B_{4R, \gamma \gamma^i}(\mathcal P_{\gamma^{i}, \xi, 4R}(x_c^{(i)})) \geq \epsilon \gamma^2 \gamma^{ni}$$ 
and for each $j=1, \dots i$ points $x_{\tilde d}^{(j)} \in \tilde B_{2R}(x_c^{(j-1)}, \gamma^{j-1})$ such that $$\vol_{g(-\gamma^2 \gamma^{2j})} \tilde B_{4R, \gamma \gamma^j}(\mathcal P_{\gamma^{j}, \xi, 4R}(x_{\tilde d}^{(j)})) < \epsilon \gamma^2 \gamma^{jn}.$$
We note that if $i=1$ then $x_c^{(i-1)}=p$ and the pinching sets are with respect to the same $\overbar W \leq \inf_{\tilde B_{8R}(p,1)}\mathcal W_{y}(\xi^{-1}).$

\item[i2)] The set $\displaystyle \tilde B_{2R}(x_c^{(i-1)}, \gamma^{i-1})\setminus \bigcup_{j=1}^{i-1} \bigcup_{\tilde d}\tilde B_{\tau^{3/2}R}(x_{\tilde d}^{(j)}, \gamma^j)\cup \bigcup_c \tilde B_{\tau^{3/2}R}(x_c^{(i-1)}, \gamma^i )$ is covered by the set $\displaystyle\bigcup_{\tilde d} \tilde B_{\tau^{3/2}R}(x_{\tilde d}^{(i)}, \gamma^i)\cup \bigcup_c \tilde B_{\tau^{3/2}R}(x_c^{(i)}, \gamma^i).$ 

\item[i3)] $(M, g(t), x_c^{(i-1)})_{[-2\delta''^{-2},0]}$ is $(n-2, \delta''^2)$--self--similar at scale $\gamma^{i-1}$ with model flow $\mathbb R ^{n-2}\times \mathbb S^2.$ 

\item[i4)] We have sets $ {\mathcal C}^{(i)} = \{ x_c^{(i)}, x_{\tilde d}^{(j)}\}$, with $ {\mathcal C}^{(i-1)}\subset {\mathcal C}^{(i)}$ and radius function $r_x^{(i)}$ such that $r_x^{(i)} = \gamma^i$ if $x = x_c^{(i)}$ and $r_x^{(i)} = \gamma^j$ if $x = x_{\tilde d}^{(j)}.$ 

\item[i5)] $\tilde B_{2R}(p,1) \subset \tilde B_{\tau R, r_x^{(i)}}({\mathcal{C}}^{(i)})$ and $(\mathcal{C}^{(i)}, r_x^{(i)})$ is a weak $(n-2, \delta')$--neck at scale in $\tilde B_{4R}(p,1)$. 
\end{itemize}

\noindent
We rewrite the covering given by $i5)$ as 
$$\tilde B_{2R}(p,1) \subset \bigcup_{j=1}^i \bigcup_{\tilde d} \tilde B_{\tau R}(x_{\tilde d}^{(j)}, \gamma^j)\cup \bigcup _c \tilde B_{\tau R}(x_c^{(i)}, \gamma^i).$$ 

We want to cover the $c$--balls of that covering. From $i1)$ and the choices we have made on $R, \gamma, \delta'',  \xi$ at step 1 we have from Lemma \ref{dicho} that $(M, g(t), x_c^{(i)})_{[-2\delta''^{-2},0]}$ is $(n-2, \delta''^2)$--self--similar at scale $\gamma^i$ with model flow $\mathbb R ^{n-2}\times \mathbb S^2.$ 

We cover the set $\displaystyle \bigcup_c\tilde B_{2R}(x_c^{(i)}, \gamma^i)\setminus \bigcup_{j=1}^{i} \bigcup_{\tilde d}\tilde B_{\tau^{3/2}R}(x_{\tilde d}^{(j)}, \gamma^j)\cup \bigcup_c \tilde B_{\tau^{3/2}R}(x_c^{(i)}, \gamma^{i+1} )$ using a maximal covering $\displaystyle \{\tilde B_{\tau^{3/2}R}(x_f^{(i+1)}, \gamma^{i+1})\}_f$ such that $\tilde B_{\tau^{5/3}R}(x_f^{(i+1)}, \gamma^{i+1})$ are pairwise disjoint. Define the set $ {\mathcal{C}}^{(i+1)} = \{x_f^{(i+1)}, x_{\tilde d}^{(j)}, x_c^{(i)}\}$ with radius function $r_x^{(i+1)}$ where $r_x ^{(i+1)}= \gamma^{i+1}$ if $x = x_f^{(i+1)}$ or $x_c^{(i)}$ and $r_x^{(i+1)} = \gamma^j$ if $x = x_{\tilde d}^{(j)}.$ We show that $(\mathcal{C}^{(i+1)}, r_x^{(i+1)})$ is a weak $(n-2,\delta')$--neck in $\tilde B_{4R}(p,1)$. 

\medskip 
For $n1)$, from $i5)$ we only have to check balls centered at $x = x_f^{(i+1)}$ and $y = x_c^{(i)}$ or $x_{\tilde d}^{(j)}$. The first case follows from construction since $d_{g(-\gamma^{2i+2})}(x, y) \geq \tau^{3/2}R \gamma^{i+1} > 2 \tau^{5/3}R \gamma^{i+1}.$ It remains to check the case $x= x_f^{(i+1)}$ and $y = x_{\tilde d}^{(j)}.$ Assume we can find $z$ such that $d_{g(-\gamma^{2i+2})}(z, x) \leq \tau^{5/3}R\gamma^{i+1}$ and $d_{g(-\gamma^{2j})}(z, y) \leq \tau^{5/3}R\gamma^j.$ Then by construction and using the triangle inequality and distance distortion \eqref{weak dist dist}, for $R > \tau^{-5/3}K$,   we have that 
$$\tau^{3/2}R\gamma^j \leq d_{g(-\gamma^{2j})} (x, y) \leq K \gamma^{j} + \tau^{5/3}R\gamma^i + \tau^{5/3}R\gamma^j < 3 \tau^{5/3}R\gamma^j,$$ 
which is a contradiction for our fixed choice of $\tau.$ 

\medskip 

We now show that $\mathcal{C}^{(i+1)} \subset \tilde B_{4R}(p,1)$. Let $x=x_f^{(i+1)}\in \mathcal{C}^{(i+1)}$ then by construction and $i1)$ we have that $x \in \tilde B_{2R}(x_c^{(i)}, \gamma^{i}),$ $x_c \in \tilde B_{2R}(x_c^{(i-1)}, \gamma^{i-1}), \dots, x_c^{(l)}\in \tilde B_{2R}(x_c^{(l-1)}, \gamma^{l-1}). $ Thus, applying distance distortion \eqref{weak dist dist}, we get that \begin{equation}\label{ensuring distances}d_{g(-\gamma^{2l-2})}(x_c^{(i)}, x_c^{(l-1)}) \leq (2R+K) \gamma^l\sum_{k = 0}^{i-l} \gamma^k,
\end{equation}
and since $K< \frac{R}{100}, \gamma < \frac{1}{10} $ we get that 
$x_c^{(i)}\in \tilde B_{4R}(x_c^{(l-1)}, \gamma^{l-1})$ for any $1\leq l\leq i.$ If we take $l=1$ then $x_c^{(l-1)} =p$ and we conclude that $x \in \tilde B_{4R}(p,1).$ 
\medskip 

To prove $n2)$, note that by $i5)$ we only have to check the property for $x= x_f^{(i+1)}$ or $x= x_c^{(i)}$ with $r_x^{(i+1)} = \gamma^{i+1}.$ By construction $x \in \tilde B_{2R}(x_c^{(i)}, \gamma^i) \cap \mathcal{L}_{x_c^{(i)}, \gamma^i}$ and $(M, g(t), x_c^{(i)})_{[-2\delta''^{-2},0]}$ is $(n-2, \delta''^2)$--self--similar at scale $\gamma^i.$ Thus from Lemma \ref{s-s/e-d}, for the same choice of $\delta''$ we made in step 1 and $\xi'>0$ the same auxiliary constant we have that 
$\mathcal W_x(\xi'^2 \gamma^2 \gamma^{2i}) - \overbar W < \xi'^2$. Since $\overbar W \leq \displaystyle \inf_{\tilde B_{8R}(p,1)} \mathcal W_y(\xi^{-1})$ and $\mathcal{C}^{(i+1)}\subset \tilde B_{4R}(p,1)$ for $\xi'< \delta'^4$ and the monotonicity of the pointed entropy we get property $n2)$ as in step 1. Furthermore, we have that 
\begin{equation}
\label{en drop 2}
\mathcal W_x(r^2 \xi') - \mathcal W_x(r^2 \xi'^{-1}) < \xi',
\end{equation}
for any $r \in [r_x^{(i+1)},\delta'^{-1}].$ 

\medskip
For $n3)$, again by $i5)$, we only have to check for $x= x_f^{(i+1)}$ or $x_c^{(i)}.$ For the same choice of $\xi'$ as in step 1 we can find $x_{(r)}\in \tilde B_R (x,r)$ such that $(M, g(t), x_{(r)})_{[-2\xi''^{-1},0]}$ is $(0,\xi'')$--self--similar at scale $r,$ where $\xi''$ is the same auxiliary constant of step 1. Assume that $r \in [\gamma^l, \gamma^{l-1}\delta'^{-1}]$ then from \eqref{ensuring distances} we have that $x_{(r)}\in \tilde B_{10R}(x_c^{(l-1)}, \max \{\gamma^{l-1}, r\}).$ If we scale by $\gamma^{l-1}$ then for $\xi''\leq \xi''(n, C_I, \Lambda, R,\gamma, \delta')$ we may apply Lemma \ref{higher order sym} to the pointed flows $(M, g(t), x_{(r)}), (M, g(t), x_c^{(l-1)})$ to get that $(M, g(t), x_{(r)})_{[-2\delta'^{-3},0]}$ is $(n-2, \delta'^3)$--self--similar at scale $r$. Hence $(M, g(t), x)_{[-2\delta'^{-2},0]}$ is $(n-2, \delta'^2)$--self--similar at scale $r$ with model flow $\mathbb R ^{n-2}\times \mathbb S^2.$ This is enough to get the inductive assumptions $(i+1)1-(i+1)5$ and finishes the induction proof.

The iteration terminates after finitely many steps since the flow is smooth at $t=0$. Indeed $ r_{Rm}^{2R}(x)$ is uniformly bounded from below, due to the smoothness of the flow,  thus after some step $i$ there will be no $c$--balls left to cover. The final set $\mathcal{C}$ is a a weak $(n-2, \delta')$--neck at scale 1 in $\tilde B_{4R}(p,1)$ that consists only of $\tilde d$--points. This finishes the proof of the lemma. 
\end{proof}

In what follows let $(\tilde {\mathcal{C}}, \tilde r_{\tilde x})$ denote the maximal weak neck we have constructed in Lemma \ref{maximal neck}.

By construction, we know that the balls $\tilde B_{\tau^{5/3}R}(\tilde x, \tilde r_{\tilde x})$, with $\tilde x \in \tilde{\mathcal C}$ are pairwise disjoint. In fact, since $\tau = 10^{-10C_I}10^{-10n}$, we also have that $\tilde B_{\tau^3 R}(\tilde x, 4\tilde r_{\tilde x})$ are pairwise disjoint. Indeed, if that was not true for $\tilde x, \tilde y\in \mathcal {\tilde C}$ with $\tilde r_{\tilde x}\leq \tilde r_{\tilde y}$ using Lemma \ref{properties of D_R}\eqref{4}, triangle inequality and distance distortion \eqref{weak dist dist}
$$\tau^{5/3}R\tilde r_{\tilde y}\leq d_{g(-\tilde r_{\tilde y}^2)}(\tilde x, \tilde y) \leq 4^{2C_I}d_{g(-16\tilde r_{\tilde y}^2)}(\tilde x, \tilde y) \leq 4^{2C_I}(4\tilde r_{\tilde x}\tau^3 R + 4\tilde r_{\tilde y}\tau^3R + K 4 \tilde r_{\tilde y}),  $$
which contradicts the choice of $\tau.$

\medskip
Since the set $\tilde {\mathcal{C}}$ is finite we define the scale distance of any $x \in M$ from $\tilde {\mathcal{C}}$ to be $\displaystyle D_{\tau ^3 R}(x, \tilde {\mathcal{C}}) = \min_{\tilde {\mathcal{C}}} D_{\tau ^3 R}(x, \tilde x)$ and let

$$r_x = \begin{cases}
    \delta^2 \tilde r_{\tilde x}, & \text{ if } D_{\tau^3R}(x, \tilde {\mathcal{C}}) = D_{\tau^3 R}(x, \tilde x) \leq \tilde r_{\tilde x,} \\
    
    \delta^2D_{\tau^3R}(x, \tilde {\mathcal{C}}), & \text{ else. } 
\end{cases}$$

\begin{LEMMA}\label{refined is almost Lipschitz} The radius function $r_x$ is well--defined and $2R$--scale $(\sigma,\delta)$--Lipschitz on $M$, for $R \geq R(n, C_I, \Lambda)$, $\sigma \in (1, \sqrt{2})$ and $\delta \leq \delta(n, C_I).$

\end{LEMMA}

\begin{proof}
For $r_x$ to be well--defined it suffices to prove that in case 1 the point $\tilde x$ where the minimum is achieved is unique. Assume there are points $\tilde x, \tilde y \in \tilde {\mathcal C}$ such that 
$D_{\tau^3R}(x, \tilde C) = D_{\tau^3R}(x, \tilde x) = D_{\tau^3 R}(x,\tilde y)$. Then for $R \geq R (n, C_I, \Lambda)$ from Lemma \ref{properties of D_R}\eqref{almost triangle} and since $\tilde B_{\tau^3 R}(\tilde x,4\tilde r_{\tilde x}) \cap \tilde B_{\tau^3R}(\tilde y, 4 \tilde r_{\tilde y})=\emptyset$ we have that for some $\sigma \in (1,\sqrt 2)$ 
$$4\tilde r_{\tilde x} \leq D_{\tau^3 R}(\tilde x, \tilde y) \leq \sigma (D_{\tau ^3R}(x, \tilde x ) + D_{\tau ^3 R}(x, \tilde y)) \leq 2 \sigma \tilde r_{\tilde x},$$
which contradicts the fact that $\sigma < 2 $ unless $\tilde x = \tilde y.$ 
\medskip
We now prove the second statement. Let $x, y $ be in case 1, namely there are $\tilde x, \tilde y \in \tilde {\mathcal C}$ such that $D_{\tau^3R}(x, \tilde {\mathcal C})=D_{\tau^3 R}(x, \tilde x)\leq \tilde r_{\tilde x}$ and $D_{\tau^3 R}(y, \tilde {\mathcal C})=D_{\tau^3 R}(y, \tilde y)\leq \tilde r_{\tilde y}$. We may assume that $\tilde x\not=\tilde y$ since otherwise $r_x=r_y$ and there is nothing to prove. Then, from Lemma \ref{properties of D_R}\eqref{almost triangle} we have that 
$$4 \tilde r_{\tilde x} \leq D_{\tau^3 R}(\tilde x, \tilde y) \leq \sigma^2 D_{\tau^3R}(x,y) + \sigma D_{\tau^3 R}(y,\tilde y)+ \sigma^2D_{\tau^3R}(x, \tilde x) \leq $$
$$\leq\sigma^2 D_{\tau^3R}(x,y) + \sigma^2 \tilde r_{\tilde x} + \sigma \tilde r_{\tilde y}.  $$
Multiplying by $\delta^2$ we get that 
$$r_x \leq \frac{\sigma^2}{4 -\sigma^2}\delta^2D_{\tau^3R}(x,y) + \frac{\sigma}{4 -\sigma^2}r_y$$ 
and since $\frac{\sigma}{4-\sigma^2}<\sigma,$ $\frac{\sigma^2}{4-\sigma^2}<\sigma$ we have that 
\begin{equation}
\label{almost Lip1} 
r_x \leq \sigma r_y + \sigma \delta^2 D_{\tau^3 R}(x,y).
    \end{equation}
This is enough since for $R \geq R(n, C_I, \Lambda)$ from Lemma \ref{properties of D_R}\eqref{sd1} $D_{\tau^3R}(x,y) \leq 4\tau^{-3}D_{2R}(x,y)$ and for $\delta \leq \delta(n,C_I)$ we get that $r_x \leq \sigma r_y + \delta D_{2R}(x,y)$ which is what we wanted.

If $x,y $ are in case 2 let $\tilde x, \tilde y\in \mathcal{\tilde C}$ such that $D_{\tau^3 R }(x, \tilde x)=D_{\tau ^3R}(x, \mathcal {\tilde C}), D_{\tau^3 R }(y, \tilde y)=D_{\tau ^3R}(y, \mathcal {\tilde C}), $ then
$$r_x = \delta^2 D_{\tau^3R}(x, \tilde x) \leq \delta^2 D_{\tau^3R}(x,\tilde y)\leq$$ $$ \leq \sigma D_{\tau^3 R}(x,y)+ \sigma D_{\tau^3 R}(\tilde y, y) = \sigma D_{\tau^3 R}(x,y) +\sigma r_y,$$
which is \eqref{almost Lip1} and is enough to get the inequality. 

If $x$ is in case 1 and $y$ in case 2 and $\tilde x, \tilde y$ points where the scale distance from $\tilde C$ is achieved respectively then  
$$r_y = \delta^2 D_{\tau^3R}(y,\tilde y)\leq D_{\tau^3 R}(y, \tilde x) \leq $$
$$\leq \sigma \delta^2 D_{\tau^3 R}(x,y) + \sigma \delta^2D_{\tau^3 R}(x,\tilde x) \leq \sigma \delta^2 D_{\tau^3 R}(x,y) + \sigma \delta^2 \tilde r_{\tilde x},$$
which again gives \eqref{almost Lip1}. Finally, if $\tilde x \neq \tilde y$ using that 
$4 \tilde r_{\tilde x} \leq D_{\tau^3 R}(\tilde x, \tilde y)$ and Lemma \ref{properties of D_R}\eqref{almost triangle}  we get again \eqref{almost Lip1} with $x$ and $y$ reversed. If $\tilde x=\tilde y$ then we have nothing to prove since then $r_x = \delta^2 \tilde r_{\tilde y} \leq \delta^2 D_{\tau^3 R}(y, \tilde y) = r_y$. This finishes the proof of lemma. 
\end{proof}

We now need to refine the set $\tilde {\mathcal{C}}$ to a new set $\mathcal{C}$ so that $(\mathcal{C}, r_x)$ also satisfies $n4), n5).$ To that end we define the following set 
$$\tilde S =  \bigcup_{\tilde x \in \tilde{\mathcal{C}}} \tilde B_{\tau R}(\tilde x, \tilde r_{\tilde x}).$$
Note that if $y \in \tilde S$ then there exists $\tilde y \in \mathcal {\tilde C}$ where the scale distance from $\mathcal {\tilde C}$ is achieved and $y \in \tilde B_R (\tilde y, \bar r_y)$ where $\bar r_y = \max \{D_{\tau ^3 R}(y, \tilde y), \tilde r_{\tilde y}  \}$. Moreover, if $\gamma < \tau ^2$ then $\bar r_y \in [\tilde r_{\tilde y},1]$.

From Lemma \ref{maximal neck} we have that $\tilde B_{2R}(p,1)\subset \tilde S$ since $\displaystyle\tilde B_{2R}(p,1) \subset \bigcup_{\tilde x \in \tilde {\mathcal{C}}}\tilde B_{\tau R}(\tilde x, \tilde r _{\tilde x})$. We need to construct a maximal covering of $\tilde S$, that contains $\tilde C$, with the variable radius function $r_x$. For that we will use the following: 

\begin{LEMMA}\label{covering}\textup{(maximal covering)} Let $(M,g(t))_I\in \mathcal{F}(n,C_I, \Lambda)$ and $A \subset S$ subsets of $M$ such that for any $x\in A$ the balls $\tilde B_{\tau^2 R}(x,r_x)$ are pairwise disjoint. Then for $R \geq R(n, C_I, \Lambda)$ there exists a set $B$ with $A \subset B$ such that $\tilde B_{\tau^{3/2}R}(x,r_x)$ with $x \in B$  cover $S$ and $B$ is maximal in terms of the condition that $\tilde B_{\tau^2R}(x,r_x)$ are pairwise disjoint. 
    
\end{LEMMA}

\begin{proof} Choose $x_1 \in S$ so that $x_1 \notin A$ and $\tilde B_{\tau^2R}(x_1, r_{x_1})$ is disjoint from $\tilde B_{\tau^2R}(x,r_x)$ for any $x \in A.$ Iteratively choose $x_l\in S$ so that $x_l \notin A$ and $\tilde B_{\tau^2 R}(x_l, r_{x_l})$ is disjoint from $\tilde B_{\tau^2 R}(x, r_x)$ for any $x\in A$ and from $\tilde B_{\tau^2 R}(x_{j}, r_{x_j})$ for $j=1,\dots, l-1$. The set  $D=\{x_1,x_2, \dots \}$ must be finite. Indeed if $l \to \infty$ then $x_l$ is a Cauchy sequence with respect to the metric $d_{g(-\rho^2)}$, where $\rho = \inf_{x \in S} r_x>0$, hence from Lemma \ref{properties of D_R}\eqref{3} we have that $x_l \in \tilde B_{\tau^2 R} (x_m,\rho)\subset \tilde B_{\tau^2 R}(x_m, r_{x_m})$ for any large $m$ and $l\geq m$, contradicting the disjoint property. 

Set $B= A \cup \{x_1, x_2, \dots, x_l\}$ we show that the balls $\tilde B_{\tau^{3/2}R}(x,r_x)$, $x \in B$, cover the set $S.$ From the maximality of the set $B$, for any $y \in S$ there exists $x \in B$  such that $\tilde B_{\tau^2 R}(x,r_x) \cap \tilde B_{\tau^2 R}(y,r_y) \neq \emptyset.$ Using the triangle inequality of Lemma \ref{properties of D_R}\eqref{almost triangle} and Lemma \ref{refined is almost Lipschitz} combined with Lemma \eqref{properties of D_R}\eqref{2} we have that 
$$(1-\sigma\delta)D_{\tau^2 R}(x,y) \leq (\sigma+\sigma^2)r_x,$$
thus 
$$D_{\tau^2 R}(x,y) \leq 4 r_x.$$ 
From Lemma \ref{properties}\eqref{4} and the choice of $\tau$ we get that $d_{g(-\tilde r_{\tilde x}^2)}(x,y) \leq 4^{2C_I}d_{g(-16\tilde r_{\tilde x}^2)}(x,y)\leq 4^{2C_I+1}\tau^2 R\tilde r_{\tilde x} < \tau^{3/2}R\tilde r _{\tilde x}, $
thus $D_{\tau^{3/2}R}(x,y)< r_x$ which is what we wanted.
\end{proof}

We now prove the following neck construction which is the analogous to the $c$--ball decomposition theorem of \cite{JiangNaber} and \cite{CJN}.
\begin{LEMMA}\label{neck construction/estimates} Let $(M, g(t), p)_{[-2\xi^{-1},0]} \in \mathcal{F}(n,C_I, \Lambda)$ such that $ r_{Rm}^{2R}(p) \leq 2$ and $M$ is orientable. If $\epsilon>0$, $R\geq R(n, C_I, \Lambda)$, $\gamma \leq \gamma(n, C_I, \Lambda, R ,\epsilon)$, $\delta \leq \delta{(n, C_I, \gamma)}$ and $\xi \leq \xi(n, C_I, \Lambda, R, \gamma, \delta, \epsilon)$ such that 
$$\vol_{g(-\gamma^2)}\tilde B_{4R, \gamma}(\mathcal P_{1, \xi, 4R}(p)) \geq \epsilon \gamma^2,$$ 
where $\overbar W \leq \displaystyle \inf_{\tilde B_{8R}(p,1)}\mathcal W_{y}(\xi^{-1})$ then there exists an $(n-2, \delta)$-neck $(\mathcal{C}, r_x)$ at scale 1 in $\tilde B_{2R}(p,1)$ such that 
$$\tilde B_R(p,1) \subset \bigcup _{ d} \tilde B_{\tau R }(x_{ d}, r_{ d})\cup\bigcup_e \tilde B_{\tau R}(x_e, r_e) \cup \bigcup_c \tilde B_{\tau R}(x_c, r_c),$$
where $x_c, x_{ d}, x_e \in \mathcal{C}$ and 
\begin{itemize}
    \item [i)] $ r_{Rm}^{2R}(x_c) \leq 2r_c $ and  $\vol_{g(-\gamma^2r_c^2)}\tilde B_{4R, \gamma r_c}(\mathcal P_{r_c, \xi, 4R}(x_c)) \geq \epsilon \gamma^2 r_c^2,$
    \item [ii)] $r_{Rm}^{2R}(x_d)\leq 2r_d$, $\vol_{g(-\gamma^2r_{ d}^2)} B_{4R, \gamma r_{ d}}(\mathcal P_{r_{ d}, \xi, 4R}(x_{ d})) < \epsilon \gamma^2 r_{ d}^2$ and $\mathcal P_{r_d, \xi, 4R}(x_d)\neq\emptyset$, 
    \item[iii)]$r^{2R}_{Rm}(x_e) \leq 2 r_e$ and $\mathcal P_{r_e, \xi, 4R}(x_e)=\emptyset.$
\end{itemize}
Furthermore we have the content estimates 
\begin{align*}   
    &\sum_{x_e\in \mathcal{C}\cap \tilde B_{\frac{3R}{2}}(p,1)}r_e^{n-2}+\sum_{x_{ d} \in \mathcal{C}\cap\tilde B_{\frac{3R}{2}}(p,1)} r_{ d}^{n-2} \leq C(n, C_I), \\ 
&\sum_{x_{c} \in \mathcal{C}\cap\tilde B_{\frac{3R}{2}}(p,1)} r_{c}^{n-2} \leq \epsilon \: C(n, C_I, \Lambda).
\end{align*}
\end{LEMMA}  

\begin{proof} From Lemma \ref{covering} there exists a finite set $\mathcal{C}$, maximal, such that $\tilde S \subset \bigcup_{x \in \mathcal C} \tilde B_{\tau^{3/2}}(x,r_x)$ and $\tilde {\mathcal{C}}  \subset \mathcal C$. We prove that the refined set $(\mathcal{C}, r_x)$ is an $(n-2, \delta)$--neck at scale 1. Clearly properties $n1), n5)$ are satisfied by the construction. 

\medskip
\noindent
For $n2)$, let $x \in \mathcal{C}$ then $x \in \tilde B_R(\tilde x, \overbar r_x)$ thus $x \in \mathcal{L}_{\tilde x, \overbar r_x} \cap \tilde B_R(\tilde x, \overbar r_x)$ and $(M, g(t), x)_{[-2\delta'^{-2},0]}$ is $(n-2, \delta'^2)$--self--similar at scale $\overbar r_x$ by construction of $\tilde {\mathcal{C}}.$ Thus from Lemma \ref{s-s/e-d} if $\delta' \leq \delta'(n, C_I, \Lambda, R, \delta, \delta_0)$, where $\delta_0$ an auxiliary constant to be fixed later, we have that $$\mathcal W_x(\delta_0^2 \delta^2 \overbar r_x^2) - \overbar W \leq \delta_0^2.$$
Since $\delta^2 \overbar r_x = r_x$ and $\xi< \delta_0 < \delta^2$ the monotonicity of the pointed entropy functional gives property $n2)$ and furthermore we have that 
\begin{equation}
\label{endrop}
\mathcal W_x(\delta_0 r^2)-\mathcal W_x(\delta_0^{-1} r^2) < \delta_0 
\end{equation}
for any $r \in [r_x, \delta^{-1}]$.    

For $n3)$ from \eqref{endrop} and Lemma \ref{s-s/e-d} there exists $x_{(r)}\in \tilde B_R(x,r)$ such that $(M, g(t), x_{(r)})_{[-2\xi_0^{-1},0]}$ is $(0, \xi_0)$--self--similar at scale $r$, provided $\delta_0 \leq \delta_0(n, C_I, \Lambda, R, \xi_0),$ where $\xi_0$ an auxiliary constant to be fixed. We now consider two cases for the scale $r.$ 

If $r \leq \overbar r_x$ then $\delta^2 \leq \frac{r}{\overbar r_x}$ and since  $(M, g(t), x_{(r)})_{[-2\xi_0^{-1},0]}$ is $(0, \xi_0)$--self--similar at scale $r$, $(M, g(t), \tilde x)_{[-2\delta'^{-2},0]}$ is $(n-2, \delta'^2)$--self--similar at scale $\overbar r_x$ and $x_{(r) }\in \tilde B_R(\tilde x, \overbar r_x)$ scaling by $\overbar r_x$ we apply Lemma \ref{higher order sym} . So we get that $(M ,g(t), x_{(r)})_{[-2\delta^{-3},0]}$ is $(n-2,\delta^3)$--self--similar at scale $r$ provided that $\xi_0 \leq \xi_0(n, C _I, \Lambda, R, \delta)$ with model flow $\mathbb R ^{n-2}\times \mathbb S^2$ which shows that  $(M ,g(t) , x)_{[-2\delta^{-2},0]}$ is $(n-2,\delta^2)$--self--similar at scale $r$.

If $r > \overbar r_x$ then $(M, g(t), \tilde x)_{[-2\delta'^{-2},0]}$ is $(n-2, \delta'^2)$--self--similar at scale $r$ and $x_{(r)} \in \tilde B_{2R}(\tilde x, r)$ thus for $\xi_0 \leq \xi_0(n, C_I, \Lambda, R, \delta)$ from Lemma \ref{higher order sym} we get that $(M, g, x_{(r)} )$ is $(n-2, \delta^2)$--self--similar at scale $r$ with model flow $\mathbb R ^{n-2}\times \mathbb S^2.$

For $n4)$ let $z_1 \in \tilde S \cap \tilde B_{2R}(x,r)$,  then there exists $z_2 \in \mathcal{C}$ such that $D_{\tau ^{3/2}R}(z_1, z_2) < r_{z_2}$. From $n5)$ and Lemma \ref{properties of D_R}\eqref{sd1} we have that 
$$r_{z_2}\leq \sigma r_{z_1} + \delta D_{2R}(z_1, z_2) \leq \sigma r_{z_1}+ \delta D_{\tau^{3/2}R}(z_1, z_2)\leq
\sigma r_{z_1}+\delta r_{z_2},$$
thus put together we have that $r_{z_2}\leq \frac{\sigma}{1-\delta}r_{z_1}.$ Again from $n5)$ we conclude that 
$$r_{z_1}\leq \sigma r_x+ \delta D_{2R}(z_1, x) \leq r(\sigma +\delta).$$ 
Putting together the above inequalities we get that 
$r_{z_2}\leq \frac{\sigma}{1-\delta}(\sigma +\delta)r < 5r$ which implies that $D_{\tau^{3/2}R}(z_1, z_2) < 5r.$ 

From Lemma \ref{properties}\eqref{4} we get that $d_{g(-r^2)}(z_1, z_2)\leq 5^{2C_I +1} \tau ^{3/2}R r < \tau^{5/4}Rr$ and so we conclude that $\tilde S \cap \tilde B_{2R}(x,r) \subset \tilde B_{\tau^{5/4}R,r}(\mathcal{C})$. Since  $\tilde B_{2R}(x,r) \subset \tilde B_{2R}(p,1)$ and $\tilde B_{2R}(p,1) \subset \tilde S$ we have that $\tilde B_{R}(x,r) \subset \tilde B_{\tau^{5/4}R,r}(\mathcal{C})$. 

Note that if we want the neck to be in $\tilde B_{2R}(p,1)$ then we can simply take $\mathcal{C}\cap \tilde B_{2R}(p,1)$ and this will now cover the ball $\tilde B_R(p,1)$ with center points in $\mathcal{C}\cap \tilde B_{2R}(p,1).$

\medskip
We now show the content estimates. Since $\displaystyle\bigcup_{x\in \mathcal{C}\cap \tilde B_{\frac{3R}{2}}(p,1)} \tilde B_{\tau ^2 R }(x,r_x) \subset \tilde B_{\frac{5R}{3}}(p,1) $, and the union is disjoint \cite[Lemma 4.1]{PanGian2} gives that
$$\sum_{x \in \mathcal{C}\cap \tilde B_{\frac{3R}{2}}(p,1)} \mu(\tilde B_{\tau^2R}(x,r_x)) \leq C(n, C_I)R^{n-2}, $$
thus
$$\sum_{x\in \mathcal{C}\cap \tilde B_{\frac{3R}{2}}(p,1)}r_x^{n-2} \leq C(n, C_I).$$  
It remains to estimate the $(n-2)$--content of $c$--balls. Define $$\mathcal{C}_c = \{x \in \mathcal{C}\cap \tilde B_{\frac{3R}{2}}(p,1) \mid x \text{ is a center of a c--ball } \}.$$  
Then from the covering of Lemma \ref{maximal neck} it follows that $\displaystyle \mu(\mathcal{C}_c) \leq \sum_{\tilde x_{\tilde d} \in \tilde{\mathcal{C}}\cap \tilde B_{\frac{28}{15}R}(p,1)}\mu(\mathcal{C}_c\cap \tilde B_{\tau R}(\tilde x_{\tilde d}, \tilde r_{\tilde d})),$ where from the maximality of $\tilde {\mathcal{C}}$ we know that each point is a center of a $\tilde d$--ball. 

Consider $x \in \mathcal{C}_c \cap \tilde B_{\tau R}(\tilde x_{\tilde d} , \tilde r_{\tilde d})$. Then from property $n5)$ and Lemma \ref{properties of D_R}\eqref{sd1} since $\tilde {\mathcal C}\subset \mathcal C$ we get that 

\begin{equation}
\label{smallness}
r_x \leq \sigma r_{\tilde x_{\tilde d}} + \delta D_{\tau R}(x, \tilde x_{\tilde d})\leq (\sigma \delta^2 + \delta)\tilde r_{\tilde d}.
\end {equation}
Since $x$ is a center of a $c$--ball we know that $\mathcal P_{r_x,\xi,4R}(x)\neq \emptyset$ so there exists $y \in \tilde B_{2R}(x, r_x)$ such that $\mathcal W_y(\xi r^2_x)-\overbar W < \xi.$ From \eqref{smallness} we have that $r_x < \tilde r_{\tilde d}$ thus monotonicity of the pointed entropy implies that the latter inequality holds at scale $\tilde r_{\tilde d}.$ Furthermore, from distance distortion \eqref{weak dist dist}, triangle inequality and \eqref{smallness} we get that 
$$d_{g(-\tilde r_{\tilde d}^2)}(y, \tilde x_{\tilde d}) \leq \tau R \tilde r_{\tilde d}+ K \tilde r_{\tilde d}+ 2R r_x < (2\tau R+ 2R(\sigma \delta^2 + \delta)) \tilde r_{\tilde d} < 2R \tilde r_{\tilde d},$$
thus in total $y \in P_{\tilde r_{\tilde d}, \xi, 4R}(\tilde x_{\tilde d})$ hence $\tilde B_{R}(y, \gamma \tilde r_{\tilde d}) \subset \tilde B_{4R, \gamma \tilde r_{\tilde d}}(P_{\tilde r_{\tilde d}, \xi, 4R}(\tilde x_{\tilde d})).$ 

Let now $z \in \tilde B_{R/2}(x, \gamma \tilde r_{\tilde d}).$ The distance distortion \eqref{weak dist dist} and triangle inequality give that 
$$d_{g(-\gamma^2\tilde r_{\tilde d}^2)}(z,y) \leq \frac{R}{2}\gamma \tilde r_{\tilde d} + K \gamma \tilde r_{\tilde d}+ 2R r_x.$$
Using \eqref{smallness} we get that 
$$d_{g(-\gamma^2\tilde r_{\tilde d}^2)}(z,y) \leq \bigg(\frac{R}2{}+ \frac{R}{100} + 2R(\sigma \delta^2 + \delta )\gamma^{-1}\bigg)\tilde r_{\tilde d}\gamma < R \gamma \tilde r_{\tilde d}$$
provided that $\delta < \delta(\gamma).$ This shows that $\tilde B_{R/2}(x, \gamma \tilde r_{\tilde d})\subset \tilde B_{4R, \gamma \tilde r_{\tilde d}}(P_{\tilde r_{\tilde d}, \xi, 4R}(\tilde x_{\tilde d})). $ 

Let a maximal cover of $\mathcal{C}_c \cap \tilde B_{\tau R}(\tilde x_{\tilde d}, \tilde r_{\tilde d})$ by balls $\tilde B_{R/2}(x, \gamma \tilde r_{\tilde d})$ with $x \in \mathcal C_c \cap \tilde B_{\tau R}(\tilde x_{\tilde d}, \tilde r_{\tilde d})$ so that $\tilde B_{R/10}(x, \gamma \tilde r_{\tilde d})$ are pairwise disjoint. Then using the non--collapsing estimate \eqref{non-coll}, the above inclusion and the fact that $\mathcal P _{\tilde r_{\tilde d}, \xi, 4R}(\tilde x_{\tilde d})$ has small $(n-2)$--content we get that the number $N$ of balls needed in this covering satisfies $N \leq C(n, C_I, \Lambda) \gamma^{2-n} \epsilon.$

From distance distortion \eqref{weak dist dist} and triangle inequality we have that $\tilde B_{2R}(x, \gamma \tilde r_{\tilde d}) \subset \tilde B_{2R}(p,1)$ thus using the covering and the upper bound of the neck structure theorem we get that
$$\mu (\mathcal{C}_c \cap \tilde B_{\tau R}(\tilde x_{\tilde d}, \tilde r_{\tilde d})) \leq \mu\bigg(\bigcup_x \tilde B_{R/2}(x, \gamma \tilde r_{\tilde d})\bigg) \leq C(n, C_I)N R^{n-2} \gamma^{n-2} \tilde r_{\tilde d}^{n-2}$$
$$\leq C(n, C_I, \Lambda)R^{n-2}\tilde r_{\tilde d}^{n-2}\epsilon.$$
From distance distortion \eqref{weak dist dist} and $K < \tau^3 R$ we get that $\tilde B_{2R}(\tilde x_{\tilde d}, \tau^3 \tilde r_{\tilde d}) \subset \tilde B_{\tau^2R}(\tilde x_{\tilde d}, \tilde r_{\tilde d})\subset \tilde B_{2R}(p,1).$
Thus since $\tilde x_{\tilde d}\in \mathcal C $ and $\tau^3 \tilde r_{\tilde d}> r_{\tilde x_{\tilde d}}$ the lower bound of the neck structure theorem applied to the ball $\tilde B_{2R}(\tilde x_{\tilde d}, \tau^3 \tilde r_{\tilde d})$ together with the above inclusion implies that
$$\mu (\mathcal{C}_c \cap \tilde B_{\tau R}(\tilde x_{\tilde d}, \tilde r_{\tilde d})) \leq  C(n, C_I, \Lambda)\epsilon \mu(\tilde B_{2R}(\tilde x_{\tilde d}, \tau^3 \tilde r_{\tilde d})) \leq C(n, C_I, \Lambda)\epsilon\mu(\tilde B_{\tau^2 R}(\tilde x_{\tilde d},\tilde r_{\tilde d})), $$
where $\tilde x_{\tilde d}\in \tilde B_{\frac{28}{15}R}(p,1).$

Putting together, we get that 
$$\mu(\mathcal{C}_c)\leq C(n, C_I, \Lambda) \epsilon \sum_{\tilde x_{\tilde d} \in \tilde{\mathcal{C}}\cap \tilde B_{\frac{28}{15}R}(p,1)} \mu(\tilde B_{\tau ^2 R}(\tilde x_{\tilde d}, \tilde r_{\tilde d})). $$
Using the fact that these balls are disjoint and that their union lies in $\tilde B_{\frac{5R}{3}}(p,1)$ we get from \cite[Lemma 4.1]{PanGian2} that 
$$\mu(\mathcal{C}_c)\leq C(n, C_I, \Lambda )\epsilon R^{n-2},$$  
which implies the desired estimates.
 \end{proof}

 \begin{REMARK} If we have a $c$--ball $\tilde B_R(x,r)$ such that $\mathcal P_{r, \xi, 4R}(x)$ is chosen with respect to $\overbar W \leq \inf_{\tilde B_{8R}(x,r)} \mathcal {W}_y(r^2 \xi^{-1})$ then a rescaling argument gives us the same covering conclusion with estimates rescaled by $r$. 
     
 \end{REMARK}

\begin{LEMMA} \label{inductive covering} Let $(M, g(t), p)_{[-2\xi^{-1},0]} \in \mathcal{F}(n,C_I, \Lambda) $, $M$ is orientable  and $R \geq R(n, C_I, \Lambda)$, $\xi \leq \xi(n, C_I, \Lambda, R)$ then 
$$\tilde B_{R}(p,1) \subset \bigcup_b \tilde B_R(x_b, r_b)\cup \bigcup_e \tilde B_R(x_e, r_e),$$
where $ r_{Rm}^{2R}(x_b) >  2r_b$ and $\mathcal P_{r_e, \xi, 4R}(x_e) = \emptyset$ with $\overbar W \leq \displaystyle \inf_{\tilde B_{8R}(p,1)}\mathcal W_{y}(\xi^{-1})$. Furthermore we have the content estimates 
$$\sum_b r_b^{n-2}+ \sum_e r_e ^{n-2} \leq C(n, C_I, \Lambda, R).$$
    
\end{LEMMA}

\begin{proof} If $\tilde B_{R}(p,1)$ is a $b$--ball or an $e$--ball then we have nothing to prove. It suffices to consider the case that $\tilde B_R(p,1)$ is a $c$--ball or a $d$--ball. Assume $\tilde B_R(p,1)$ is a $c$--ball, namely $ r^{2R}_{Rm}(x) <2 $ and $\vol_{g(-\gamma^2)}(\tilde B_{4R, \gamma}(\mathcal P_{1, \xi, 4R}(p)) \geq \epsilon \gamma^2$. The argument is analogous if $\tilde B_R(p,1)$ is a $d$--ball. Let $R \geq R(n, C_I, \Lambda)$ , $\epsilon\leq \epsilon (n, C_I, \Lambda, R)$, $\gamma \leq \gamma(n, C_I, \Lambda, R, \epsilon)$, $\delta \leq \delta(n, C_I,\gamma)$, $\xi \leq \xi(n, C_I, \Lambda, R, \epsilon, \gamma, \delta)$ as in Lemma \ref{neck construction/estimates} and Lemma \ref{d-ball dec}. 

Applying Lemma \ref{neck construction/estimates} we have that 
\begin{equation}
\label{covering1}
\tilde B_R(p,1)\subset \bigcup_c \tilde B_R(x_c, r_c)\cup \bigcup _d \tilde B_R(x_d, r_d)\cup \bigcup_e \tilde B_R(x_e,r_e),
\end{equation}
with 
\begin{equation}\label{est1}
 \begin{aligned}   
\sum_d r_d^{n-2}+\sum_e r_e^{n-2} &\leq C(n, C_I) \\
\sum _c r_c^{n-2}&\leq C(n, C_I, \Lambda)\epsilon.
\end{aligned}
\end{equation}

To each $d$--ball we apply Lemma \ref{d-ball dec} thus
\begin{equation}\label{cov2}
\tilde B_{R}(x_d,r_d) \subset \bigcup_b \tilde B_R(x_b,r_b)\cup \bigcup _c \tilde B_R(x_c,r_c) \cup \bigcup_e \tilde B_R(x_e,r_e) 
\end{equation}
with estimates 
\begin{equation}\label{est2}
\begin{aligned}
\sum_b r_b^{n-2} +\sum_e r_e^{n-2} &\leq C(n, C_I, \Lambda, R, \gamma)r_d^{n-2},   \\
\sum_cr_c^{n-2}&\leq C(n, C_I, \Lambda, R)r_d^{n-2}\epsilon.
\end{aligned}
\end{equation}
Putting together \eqref{covering1}, \eqref{cov2}, \eqref{est1} and \eqref{est2} we get that 
$$\tilde B_R (p,1) \subset \bigcup_b \tilde B_R(x_b^{(1)}, r_b^{(1)})\cup \bigcup_c\tilde B_R(x_c^{(1)}, r_c^{(1)})\cup \bigcup_e\tilde B_{R}(x_e^{(1)},r_e^{(1)}),$$
with 
$$\sum_b (r_b^{(1)}) ^{n-2} \leq C(n, C_I, \Lambda, R, \gamma)C(n, C_I, \Lambda, R),$$
$$\sum_e(r_e^{(1)})^{n-2} \leq \overbar C(n, C_I, \Lambda, R, \gamma),$$
$$\sum_c (r_c^{(1)})^{n-2} \leq \overbar C(n, C_I, \Lambda, R)\epsilon,$$ 
where the constants are chosen with respect to the combined constants from estimates of Lemma \ref{neck construction/estimates} and Lemma \ref{d-ball dec}. 

Assume that we have applied the combination of Lemma \ref{neck construction/estimates} and Lemma \ref{d-ball dec} $i$--times and we have 
$$\tilde B_R(p,1) \subset \bigcup_{j=1}^i \bigcup_b \tilde B_R(x_b^{(j)}, r_b^{(j)}) \cup \bigcup _e\tilde B_R (x_e^{(j)}, r_e^{(j)})\cup \bigcup_c \tilde B_R(x_c^{(i)}, r_c^{(i)}),$$
with estimates $$\sum_{j=1}^i \sum_b (r_b^{(j)})^{n-2} \leq C(n, C_I, \Lambda, R, \gamma)C(n, C_I, \Lambda, R) \sum_{j=0}^{i-1}(\overbar C(n, C_I, \Lambda, R) \epsilon)^{j}, $$
$$\sum_{j=1}^i\sum_e (r_e^{(j)})^{n-2} \leq \overbar C(n, C_I, \Lambda, R, \gamma )\sum_{j=0}^{i-1}(\overbar C(n, C_I, \Lambda, R) \epsilon)^{j},$$
$$\sum_c (r_c^{(i)})^{n-2} \leq (\overbar C(n, C_I, \Lambda, R)\epsilon)^{i},$$
where the constants are as in step 1. Then we apply Lemma \ref{neck construction/estimates} to each $c$--ball which gives balls with content estimates 
$$\sum_e (r_e^{(i+1)})^{n-2} + \sum_d (r_d^{(i+1)})^{n-2} \leq C(n, C_I) (\overbar C(n, C_I, \Lambda, R)\epsilon)^i,$$
$$\sum_c (r_c^{(i+1)})^{n-2} \leq C(n, C_I, \Lambda)\epsilon (\overbar C(n, C_I, \Lambda, R)\epsilon )^i.$$
To each $d$--ball that we get we apply Lemma \ref{d-ball dec} to get extra $e$--balls, $b$--balls, $c$--balls that satisfy the estimates 
$$\sum_e (r_e^{(i+1)})^{n-2} + \sum_b (r_b^{(i+1)})^{n-2} \leq C(n, C_I, \Lambda, R, \gamma)C(n, C_I) (\overbar C(n, C_I, \Lambda, R)\epsilon)^i,$$
$$\sum_c (r_c^{(i+1)})^{n-2}\leq C(n, C_I, \Lambda, R) \epsilon C(n, C_I) (\overbar C(n, C_I, \Lambda, R)\epsilon)^i.$$ 
If we put together the $(n-2)$--content of the new $c$--balls is 
$$\sum_c (r_c^{(i+1)})^{n-2}\leq (\overbar C(n, C_I, \Lambda, R )\epsilon)^i \epsilon \bigg( C(n, C_I, \Lambda, R)C(n, C_I)+$$ $$+C(n, C_I, \Lambda))\bigg)  
=(\overbar C(n, C_I, \Lambda, R)\epsilon)^{i+1}.$$ 
Similarly for the $b$--balls and $e$--balls we have that 
$$\sum_{j=1}^{i+1} \sum_b  (r_b^{(j)})^{n-2} \leq C(n, C_I, \Lambda, R, \gamma)C(n, C_I, \Lambda, R)\sum_{j=0}^{i}(\overbar C(n, C_I, \Lambda, R)\epsilon)^j.$$

$$\sum_{j=1}^{i+1}(r_e^{(j)}) ^{n-2} \leq \overbar C(n, C_I, \Lambda, R, \gamma)\sum_{j=0}^{i-1}(\overbar C(n, C_I, \Lambda, R)\epsilon )^{j}+$$
$$+ (\overbar C(n, C_I, \Lambda, R)\epsilon)^i\bigg(C(n, C_I, \Lambda, R, \gamma)C(n, C_I, \Lambda, R) + C(n, C_I)\bigg)=$$
$$\overbar C(n, C_I, \Lambda, R, \gamma)\sum_{j=0}^i (\overbar C(n, C_I, \Lambda, R)\epsilon)^j,$$ 
which finishes the induction step. 

The process must terminate after finitely many steps since the flow is smooth. Fixing $\epsilon = \epsilon (n, C_I, \Lambda, R)$ so that the series in the estimates converge and $\gamma= \gamma(n, C_I, \Lambda, R, \epsilon)$, $\delta = \delta(n, C_I, \gamma)$ gives the proof of the lemma. 
\end{proof}

\begin{LEMMA}\label{v-ball dec} Let $(M, g(t), p)_{[-2,0]}\in \mathcal{F}{(n,C_I, \Lambda)}$, $M$ orientable and $R \geq R(n,C_I, \Lambda)$. Then there exists $w_0(n, C_I, \Lambda, R)>0$ such that 
$$\tilde B_R(p,1) \subset \bigcup_b \tilde B_R (x_b, r_b)\cup \bigcup_w \tilde B_R (x_w, r_w),$$
where $ r_{Rm}^{2R}(x_b)>  2r_b$, $\displaystyle\inf_{\tilde B_{2R}(x_w, r_w)}\mathcal W_x(r_w^2) \geq \inf_{\tilde B_{2R}(p,1)}\mathcal W_y(1) + w_0$ and content estimates 
$$\sum_b r_b^{n-2} + \sum_w r_w^{n-2}\leq C(n, C_I, \Lambda, R).$$
    
\end{LEMMA}

\begin{proof} Choose a maximal covering of $\tilde B_{R}(p,1)$ by balls $\{\tilde B_R (x_f, \sqrt \xi)\}_{f=1}^N$ where $x_f \in \tilde B_R(p,1)$ and the number of balls $N \leq C(n, C_I, \Lambda, R, \xi)$ (we have seen how to construct such coverings and the bound on $N$ follows from the non--collapsing and the non--inflating estimates \eqref{non-coll}, \eqref{non-inf} respectively). We will  apply Lemma \ref{inductive covering} to each $\tilde B_R(x_f, \xi)$ fixing $\overbar W = \inf_{\tilde B_{2R}(p,1)}\mathcal W_y(1).$ 

Since $\tilde B_{8R}(x_f, \sqrt \xi) \subset \tilde B_{2R}(p,1)$ (for $\xi$ small enough) it follows that $\inf_{\tilde B_{8R}(x_f, \sqrt \xi)}\mathcal W_y(1) \geq \inf_{\tilde B_{2R}(p,1)}\mathcal W_y(1) $. Therefore, if $R \geq R (n, C_I,\Lambda)$ and $\xi \leq \xi(n, C_I, \Lambda, R)$ we can apply Lemma \ref{inductive covering} at scale $\sqrt \xi$ to obtain a covering

$$\tilde B_R(x_f, \sqrt{\xi}) \subset \bigcup_b \tilde B_R (x_b, r_b) \cup\bigcup _e \tilde B_R (x_e, r_e),$$
where $ r_{Rm}^{2R}(x_b)> 2r_b$ and $\mathcal P_{r_e, \xi, 4R}(x_e) = \emptyset$, that satisfies the content estimates

$$\sum_b r_b^{n-2}+ \sum_e r_e^{n-2}\leq C(n, C_I, \Lambda, R).$$
Since $\mathcal P_{r_e, \xi, 4R}(x_e) = \emptyset$, for any $x \in \tilde B_{2R}(x_e, r_e) $ we know that
$$\mathcal W_x(r_e^2\xi)\geq  \overbar{W} + \xi = \inf_{\tilde B_{2R}(p,1)}\mathcal W_y(1) + \xi. $$

Set $\xi = \xi(n, C_I, \Lambda, R)$, $w_0 = \xi$ and choose a maximal covering for each $\tilde B_R (x_e, r_e)$ by $\tilde B_R(x_e^{(j)}, r_e \sqrt\xi) \subset \tilde B_{2R}(x_e, r_e)$ for each $j=1, \dots, L$. Then from the non--collapsing \eqref{non-coll} and the non--inflating \eqref{non-inf} we get that $L \leq C(n, C_I, \Lambda, R)$ and by construction we also have
$$\inf_{\tilde B_{2R}(x_e^{(j)}, r_e\sqrt \xi)}\mathcal  W_x(r_e^2\xi) \geq \inf_{\tilde B_{2R}(x_e, r_e)} \mathcal W_x(r_e^2 \xi)\geq \:  \overbar W + \xi.$$
Thus setting $r_w = r_e\sqrt \xi$ we get the desired decomposition. 
\end{proof}

 Applying recursively Lemma \ref{v-ball dec}, we show that after finitely many steps there will be no more $w$-balls left, due to the lower entropy bound and Perelman's Harnack inequality \eqref{PH in}. This suffices to complete the proof of Theorem \ref{b-ball dec} below. 

\begin{proof}[Proof of Theorem \ref{b-ball dec}] We apply Lemma \ref{v-ball dec} to $\tilde B_R(p,1)$ to obtain a covering
$$\tilde B_R(p,1)\subset \bigcup_b\tilde B_R(x_b^{(1)}, r_b^{(1)})\cup \bigcup_w \tilde B_R(x_w^{(1)}, r_w^{(1)}),$$
    with $ r_{Rm}^{2R}(x_b^{(1)}) > 2r_b^{(1)}$ and 
\begin{equation*}
\begin{aligned}
\inf_{\tilde B_{2R}(x_w^{(1)}, r_w^{(1)})}\mathcal  W_x((r_v^{(1)})^2) &\geq \inf_{\tilde B_{2R}(p,1)} \mathcal W_y (1)+ w_0,\\
\sum_b (r_b^{(1)})^{n-2}+ \sum_w (r_w^{(1)})^{n-2} &\leq C(n, C_I, \Lambda, R).
\end{aligned}
\end{equation*}
Now, iteratively appling Lemma \ref{v-ball dec} to each $w$-ball,  we obtain after $i$--steps a covering
$$\tilde B_R (p,1) \subset \bigcup_{j=1}^i \bigcup _b \tilde B_R (x_b^{(j)}, r_b^{(j)}) \cup \bigcup_w \tilde B_R (x_w^{(i)}, r_w^{(i)}),$$ 
with 
\begin{equation*}
\begin{aligned}
\sum_{j=1}^i \sum_b (r_b^{(j)})^{n-2} \leq i (C(n, C_I, \Lambda, R))^i,\\
\sum_w (r_w^{(i)})^{n-2}\leq (C(n, C_I, \Lambda, R))^i,\\
\inf_{\tilde B_{2R}(x_w^{(i)}, r_w^{(i)})} \mathcal W_x((r_w^{(i)})^2) \geq \inf_{\tilde B_{2R}(p,1)} \mathcal W_y(1) + i w_0.
\end{aligned}
\end{equation*}

In particular, if $i \geq \frac{2\Lambda}{w_0}$ then 
$$\inf_{\tilde B_{2R}(x_w^{(i)}, r_w^{(i)})} \mathcal W_x((r_w^{(i)})^2) \geq \Lambda > 0,$$
which contradicts Perelman's Harnack inequality. Thus the process terminates after at most $i=i(n,C_I,\Lambda, R)$ steps, which suffices to prove the result.
\end{proof}

\section{Applications of the decomposition theorem}

\inlinesubsection{Curvature bounds} We use the decomposition theorem to prove $L^1$--bounds on the curvature tensor. We prove in fact the following stronger bound on the curvature radius from which Theorem \ref{intro:thm} immediately follows.   

\begin{theorem}\label{curv radius int} Let $(M, g(t), p)_{[-2,0]}\in \mathcal{F}(n,C_I, \Lambda)$ then for $R \geq R(n, C_I, \Lambda)$
$$\int_M ( r_{Rm}^{2R}(x))^{-2} dV_{g(0)} \leq C(n, C_I, \Lambda, R, \vol_{g(-1)}(M)).$$
\end{theorem}

\begin{proof}  We may assume that $M$ is orientable since else we can pass to the orientable double cover. Consider a maximal covering of $M$ by $\tilde B_R (x_f,1)$ where $R \geq R(n, C_I, \Lambda)$ so that Theorem \ref{b-ball dec} holds. From the non--collapsing \eqref{non-coll} the number of balls $N$ in this covering satisfies $N \leq v_1^{-1}\vol_{g(-1)}(M)$. To each ball $\tilde B_R(x_f,1) $ we may apply Theorem \ref{b-ball dec} to obtain a covering 
$$\tilde B_{R}(x_f, 1) \subset \bigcup_b \tilde B_{R}(x_b, r_b),$$
with $r_{Rm}^{2R}(x_b) > 2r_b$ and 
$$\sum_b r_b^{n-2}\leq C(n, C_I, \Lambda, R).$$
Then we have that
$$\int_M ( r_{Rm}^{2R}(x))^{-2}dV_{g(0)} \leq \sum_f \int_{\tilde B_{R}(x_f,1)} ( r_{Rm}^{2R}(x))^{-2}dV_{g(0)} \leq N\sum_b \int_{\tilde B_{R}(x_b,r_b)}(r_{Rm}^{2R}(x))^{-2}dV_{g(0)}.$$
Since the curvature radius function is $2R$--scale $(\sigma, \sigma)$--Lipschitz for $\sigma < \frac{4}{3}$, see \cite[Lemma 5.1]{PanGian2}, we have that for any $x\in \tilde B_R (x_b, r_b)$
$$ r_{Rm}^{2R}(x_b) \leq \sigma  r_{Rm}^{2R}(x) + \sigma D_{2R}(x,x_b)<\sigma  r_{Rm}^{2R}(x) +\sigma r_b. $$
Since $ r_{Rm}^{2R}(x_b) > 2r_b$ we get that 
    $$ r_b <  2r_{Rm}^{2R}(x),$$ 
for any $x \in \tilde B_R(x_b,r_b).$
Putting together, we have that 
$$\int_M ( r_{Rm}^{2R}(x))^{-2}dV_{g(0)} \leq C(n, C_I, \Lambda, \vol_{g(-1)}(M))  \sum_br_b^{-2}\vol_{g(0)}\tilde B_{R}(x_b,r_b).$$
Using \eqref{volume comp} and the non--inflating estimate \eqref{non-inf} (since $Rr_b <1$ by construction) together with the content estimates of Theorem \ref{b-ball dec} we get that   $$\int_M ( r_{Rm}^{2R}(x))^{-2}dV_{g(0)} \leq C(n, C_I, \Lambda, R, \vol_{g(-1)}(M))\sum_b r_b^{n-2} \leq C(n, C_I, \Lambda, R, \vol_{g(-1)}(M)),$$
    which shows the desired integral estimate. 
\end{proof}

\noindent Now we can apply Theorem \ref{curv radius int} to finish the proof of Theorem \ref{intro:thm}.

\begin{proof}[Proof of Theorem \ref{intro:thm}]
The proof of Theorem \ref{intro:thm} follows as in the analogous case in \cite{PanGian2}, by appropriately rescaling and translating in time the estimate of Theorem \ref{curv radius int}.
\end{proof}

\noindent
\inlinesubsection{Content estimates of the top quantitative stratum} Define

$$\mathcal{S}^k_{\eta,r^2} = \bigg\{x \in M \mid (M,g(t),x) \text{ is not } (k+1, \eta)\text{--self}\text{--similar at scale } s \text{ for any } s\in[r,1]  \bigg\}.$$ 
If $k = n-2$ then we have the top quantitative stratum 
$$\mathcal{S}^{n-2}_{\eta,r^2}=\bigg\{x \in M \mid (M,g(t),x) \text{ is not } (n, \eta) \text{--self}\text{--similar } \text{ at scale } s \text{ for any } s\in[r,1]  \bigg\}.$$ 
We prove the following content estimate.

\begin{theorem}\label{content est} Let $(M,g(t),p)_{[-2,0]}\in \mathcal{F}(n,C_I, \Lambda)$, $R \geq R(n, C_I, \Lambda)$, $\eta, r \in(0,1)$ then 
$$\vol_{g(0)} ( \mathcal{S}^{n-2}_{\eta,r^2}\cap \tilde B_R(p,1)) \leq C(n, C_I, \Lambda, R, \eta) r^2. $$
\end{theorem}
\noindent
In order to do that we will need the following gap lemma for the scalar curvature. 

\begin{LEMMA}\textup{(Gap Lemma)} \label{gap lemma}Fix $\zeta>0$ and assume that $(M, g(t), z)\in \mathcal{F}(n,C_I, \Lambda)$ is $(0, \eta')$--self--similar at scale 1 and not $(n, \eta^2)$--self--similar at scale 1. Then if $\eta' \leq \eta'(n, C_I, \Lambda, \eta, \zeta)$ there exists $\epsilon (n, C_I)>0 $ such that $|\text{Scal}|_{g(t)}(z) \geq \dfrac{\epsilon}{|t|}$ for all $t \in (-\zeta^{-1}, -\zeta).$ 
    
\end{LEMMA}
 
\begin{proof} We argue by contradiction.Assume that there is a sequence $\eta'_j \to 0$ and Ricci flows $(M_j, g_j(t), z_j)\in \mathcal{F}(n,C_I, \Lambda)$ that are $(0,\eta'_j)$--self--similar and not $(n,\eta^2)$--self--similar at scale 1, and there are $t_j \in (\zeta^{-1},-\zeta)$ such that $|t_j|\:|\text{Scal}|_{g_j(t_j)}(z_j) \to 0.$ From Theorem \ref{compactness}\eqref{II} and for $\eta'_j \leq \eta'(n, C_I, \Lambda, \eta, \zeta)$ we obtain a limit $0$--self--similar flow  $(M_\infty, g_\infty(t), z_\infty)$ that is not $(n, \frac{\eta^2}{2})$--self--similar. On the other hand, up to a subsequence $t_j \to t_\infty$, and $\text{Scal}_{g_\infty(t_\infty)}(z_\infty)=0$. Since $z_\infty \in S_\text{point}$ we conclude that $\text{Scal}_{g(t)}(z_\infty) =0$ for all $t<0$, and by the strong maximum principle we conclude that the scalar curvature vanishes everywhere. 

From the evolution equation for the scalar curvature we conclude that the flow is Ricci flat, so the gradient shrinking Ricci soliton equation gives that the soliton functions $f_t$, $t<0$, satisfy $\nabla^2_{g(t)}f_t = \frac{1}{2|t|}g(t)$ for all $t<0.$ The latter is enough to imply that the flow is the Gaussian soliton which is a contradiction. 
\end{proof}

\begin{proof}[Proof of Theorem \ref{content est}] We may assume that $M$ is orientable. From Theorem \ref{b-ball dec} for $R \geq R(n, C_I, \Lambda)$ we obtain a covering of the form $\tilde B_R(p,1) \subset \bigcup_b \tilde B_R(x_b, r_b) $, where $ r_{Rm}^{2R}(x_b) >2 r_b$ and $\sum_b r_b^{n-2}\leq C(n,C_I, \Lambda, R)$. In order to prove the estimate we have to investigate the intersections $\mathcal{S}^{n-2}_{\eta,r^2} \cap \tilde B_R(x_b, r_b)$.

Assume that $ r\leq \alpha r_b$, where $0<\alpha<1$ will be fixed later, and let $y\in \mathcal{S}^{n-2}_{\eta,r^2} \cap \tilde B_R(x_b, r_b)$. By the monotonicity of the pointed entropy at $y$ and the lower bound on the $\nu$-entropy, we conclude that for any $\theta>0$ and $\alpha \leq \alpha(\theta)$ there exists $\bar r \in [\alpha r_b, r_b]$ such that
$$\mathcal W_y(\bar r^2 \theta)-\mathcal W_y(\bar r^2 \theta^{-1}) <\theta.$$ 
Thus, if $\theta \leq \theta (n, C_I, \Lambda, R, \eta')$, by Lemma \ref{s-s/e-d} there exists $z \in \tilde B_{R/2}(y, \bar r)$ so that $(M, g(t), z)$ is $(0, \eta')$--self--similar at scale $\bar r.$ 

If $(M, g(t), z)$ is $(n, \eta^2)$--self--similar at scale $\bar r$ then $(M,g(t),y)$ is $(n,\eta)$--self--similar at scale $\bar r$ (since in that case $S_{\text{point}}=\mathbb R^n$). Since $r\leq \bar r$ the latter contradicts the fact that $(M, g(t), y)$ is not $(n,\eta)$--self--similar at any scale $[r,1].$ Thus $(M, g(t), z)$ is not $(n, \eta^2)$--self--similar at scale $\bar r$. Let $\zeta>0$ to be fixed later and fix $\eta'\leq \eta'(n, C_I, \Lambda, R, \eta, \zeta)$ so that by Lemma \ref{gap lemma} there exists $\epsilon(n, C_I)>0$ such that $|\text{Scal}|_{g(t)}(z) \geq \frac{\epsilon}{|t|}$ for all $t \in (-\zeta^{-1}\bar r^2, -\zeta\bar r^2).$
Since $z \in \tilde B_{2R}(x_b, r_b)$ and the latter is a $b$--ball we also have that 
$$|\text{Scal}_{g(t)}|(z) \leq \frac{C(n)}{r_b^2},$$
for al $t \in [-r_b^2,0].$
Put together, if we take $t=-\zeta \bar r^2$ we arrive at a contradiction provided that $\zeta < C(n)\epsilon.$ In conclusion, if we fix $\alpha:=\alpha(n, C_I, \Lambda, R, \eta)$ then we have shown that for $r \leq \alpha r_b$ 
$$\mathcal{S}^{n-2}_{\eta,r^2}\cap\tilde B_{R}(x_b,r_b)=\emptyset.$$ 

Using the previous conclusion along with \eqref{volume comp}, the non--inflating \eqref{non-inf} and the $(n-2)$--content estimates we have that
$$\vol_{g(0)}(\mathcal{S}^{n-2}_{\eta,r^2}\cap \tilde B_R(p,1)) \leq \sum_{r_b<\alpha^{-1}r} \vol_{g(0)}( \mathcal{S}^{n-2}_{\eta,r^2}\cap \tilde B_R(x_b,r_b))$$ $$\leq C(n, C_I, \Lambda) \sum_{r_b < \alpha^{-1}r} r_b^n \leq  C(n, C_I, \Lambda, R, \eta) r^2, $$
which shows the desired estimate.
\end{proof}

\begin{small}

\end{small}
\end{document}